\documentclass[reqno,12pt,letterpaper]{amsart}
\usepackage[proof]{newmzmacros1}
\usepackage[all,cmtip]{xy}
\newcommand{\RR}{{\mathbb R}}
\newcommand{\NN}{{\mathbb N}}
\newcommand{\CC}{{\mathbb C}}

\newcommand{\CI}{{C^\infty}}

\newcommand{\TT}{{\mathbb T}}
\newcommand{\e}{\epsilon}
\usepackage{tikz}
\usetikzlibrary{angles,quotes,math}
\usepackage{xcolor}
\definecolor{light-gray}{rgb}{.5,.5,.5}

\title{Analytic hypoellipticity of Keldysh operators}

\author{Jeffrey Galkowski}
\email{j.galkowski@ucl.ac.uk}
\address{Department of Mathematics, University College London, WC1H 0AY, UK}

\author{Maciej Zworski} 
\email{zworski@math.berkeley.edu}
\address{Department of Mathematics, University of California, Berkeley, CA 94720}

\begin{document}

\begin{abstract}
We consider Keldysh-type operators, $ P = x_1 D_{x_1}^2 + a (x) D_{x_1} +  Q (x, D_{x'} ) $, $ x = ( x_1, x') $ with analytic coefficients, 
and with $ Q ( x, D_{x'} ) $ second order, principally real  and elliptic in $ D_{x'} $ for 
$ x $ near zero.
We show that if $ P u =f $, $ u \in C^\infty $, and $ f $ is {\em analytic} in a neighbourhood of $ 0 $ then $ u $ is {\em analytic} in a neighbourhood of 
$ 0 $. This is a consequence of a microlocal result valid for operators of any order with Lagrangian radial sets. Our result proves a generalized version of a conjecture made in \cite{revres}, \cite{LZ} and has applications to scattering theory.
\end{abstract}

\maketitle

\section{Introduction}
\label{s:intr}

We consider analytic regularity for generalizations of the Keldysh operator \cite{Keld}, 
\begin{equation}
\label{eq:Keld0} P :=  x_1 D_{x_1}^2 + D_{x_2}^2  .
\end{equation}
The operator $ P $ has the feature of changing
from an elliptic to a hyperbolic operator at $ x_1 = 0 $. It appears in various places including the study of transsonic flows, see for instance \v{C}ani\'c--Keyfitz 
\cite{cake} or population biology -- see Epstein--Mazzeo 
\cite{epm}. Our interest in such operators comes from the work of 
Vasy \cite{vas} where the transition at $ x_1 = 0 $ corresponds to the
boundary at infinity for asymptotically hyperbolic manifolds (see \cite{V4D}), crossing the event horizons of Schwartzschild black holes (see 
\cite[\S 5.7]{dizzy}) or the cosmological horizon for de Sitter spaces.
The Vasy operator in the asymptotically hyperbolic setting is given by 
\begin{equation}
\label{eq:Plag} P ( \lambda ) = 4 ( x_1 D_{x_1}^2 - ( \lambda + i ) D_{x_1} )
- \Delta_{h(x_1)}  + i \gamma ( x ) \left( 2 x_1 D_{x_1} - \lambda - i
{\textstyle\frac{ n-1} 2 } \right) , \end{equation}
where $ h(x_1 ) $ is a smooth family of Riemannian metrics in $ x' $, 
$ x = ( x_1, x') \in \RR^n $ and $ \gamma \in \CI ( \RR^n ) $. The resonant states at 
resonant frequencies $ \lambda $ 
(see \cite[Chapter 5]{dizzy}) are the smooth 
solutions of $ P ( \lambda ) u = 0 $. 

For various reasons reviewed in 
\S \ref{s:soma} it is interesting to ask if in the case of analytic coefficients the resonant states are real analytic across $ x_1 = 0 $.
That lead to \cite[Conjecture 2]{revres} which asked if $ P ( \lambda ) 
u = f $ with $ u $ smooth and $ f $ analytic near $ x_1 = 0 $ implies that
$ u $ is analytic near $ x_1 = 0 $. For $ \gamma ( x ) \equiv 0 $
and $ h $ independent of $ x_1 $, this was shown by Lebeau--Zworski
\cite{LZ} under the assumption that $ \lambda \notin - i \NN^* $. 

The general case was proved by Zuily \cite{zulu} under the same 
restriction on $ \lambda $. His proof was an elegant adaptation of
 the work of
Baouendi--Goulaouic \cite{BG}, Bolley--Camus \cite{bc} and
Bolley--Camus--Hanouzet \cite{bcz}. 

In this paper we prove this result for generalized Keldysh operators
with analytic coefficients \eqref{eq:Keld}. In particular, we do not make any assumptions on lower order terms:
\begin{theo}
\label{t:1}
Suppose that $ U\subset \RR^n  $ is a neighbourhood of $ 0 $,
\begin{equation}
\label{eq:Keld}   P := x_1 D_{x_1}^2 + a (x) D_{x_1} +  Q (x, D_{x'} ) , \ \ 
\  x = ( x_1, x') \in U ,\end{equation}
 has analytic coefficients, 
$ Q ( x, D_{x'} ) $ is a second order elliptic operator 
in $ D_{x'} $ with a real valued principal symbol. Then there exists
a neighbourhood of $ 0 $, $ U' \subset U $, such that
\begin{equation}
\label{eq:t1}   P u \in C^\omega ( U ) , \ \  u \in C^\infty (U) \ \Longrightarrow \ 
u \in C^\omega ( U' ) . \end{equation} \end{theo}
We will show in \S \ref{s:micro} that this result follows from a more general microlocal result valid for operators of all orders satisfying a natural geometric condition.

\noindent
{\bf Remarks:} 1. In the statement of the theorem $ 0 $ can be replaced by any point at which $ x_1 \geq 0 $ and $ U' $ can be replaced by $ U$ provided we include a bicharacteristic convexity condition. That follows from propagation of analytic singularities -- see 
\cite[Theorem 4.3.7]{M} or \cite[Theorem 2.9.1]{his}: since there are no singularities near $ x_1 = 0 $ there will be no singularities on 
trajectories hitting $ x_1 = 0 $ -- see Figure~\ref{f:radial1}.

\begin{figure}
\begin{tikzpicture}
\begin{scope}[scale=2]
\begin{scope}[shift={(-4,0)}]
\coordinate (C) at (1,0);
\tikzmath{ \m=4;\aShift=1/2; \maxx=1/3;\N=8;}
\draw (-1.5,1)--(2,1);
\draw (-1.5,-1)--(2,-1);
\begin{scope}[shift=(C)]

\foreach \i in {1,...,\N} {
\draw[->,color=light-gray,dashed,thick]   plot[smooth,domain=\maxx*(((\i-1)/\N)*sqrt((\i-1)/\N)-.05):\maxx*((\i/\N)*sqrt(\i/\N))] ( {-\m*tan (\x*90)* tan(\x*90) +cos ((\x+\aShift)*90)/2},{sin ((\x+\aShift)*90)});
\draw[->,color=light-gray,dashed,thick]   plot[smooth,domain=-\maxx*(((\i-1)/\N)*sqrt((\i-1)/\N)-.05):-\maxx*((\i/\N)*sqrt(\i/\N))] ( {-\m*tan (\x*90)* tan(\x*90) +cos ((\x+\aShift)*90)/2},{sin ((\x+\aShift)*90)});
}

\draw[color=light-gray,dashed]   plot[smooth,domain=-1:1] ( {cos ((\x)*90)/2},{sin ((\x)*90)});
\draw[color=light-gray,dashed]   plot[smooth,domain=-1:1] ({-1.8+cos ((\x)*90)/2},{sin ((\x)*90)});

\draw[black]   plot[smooth,domain=1:3] ( {cos ((\x)*90)/2},{sin ((\x)*90)});

\draw[black]   plot[smooth,domain=1:3] ( {-1.8+cos ((\x)*90)/2},{sin ((\x)*90)});

\foreach \i in {1,...,\N} {
\draw[black,->,thick]   plot[smooth,domain=\maxx*((\i/\N)*sqrt(\i/\N)+.05):\maxx*((\i-1)/\N)*sqrt((\i-1)/\N)] ( {-\m*tan (\x*90)* tan(\x*90)+cos ((\x+2+\aShift)*90)/2},{sin ((\x+2+\aShift)*90)});
\draw[black,->,thick]   plot[smooth,domain=-\maxx*((\i/\N)*sqrt(\i/\N)+.05):-\maxx*((\i-1)/\N)*sqrt((\i-1)/\N)] ( {-\m*tan (\x*90)* tan(\x*90)+cos ((\x+2+\aShift)*90)/2},{sin ((\x+2+\aShift)*90)});
}

\draw[shift=(C)] node[below] at (0,-1.2){\tiny{$x_1=0$}};
\draw[shift=(C)] node[below] at (-2,-1.2){\tiny{$x_1<0$}};
\draw[shift=(C)] node[below] at (-1,-1.5){{Keldysh}};

\draw node[circle,fill, inner sep =2pt, label= right:{\tiny{$\theta=0$}}] at ({cos ( (0+\aShift) *90)/2} , {sin((0+\aShift)*90)}){};
\draw node[circle,fill, inner sep =1pt, label= right:{\tiny{$\theta=\tfrac{ \pi}{2}$}}] at ({cos ( (1+\aShift) *90)/2} , {sin((1+\aShift)*90)}){};
\draw node[circle,fill, inner sep =2pt, label= right:{\tiny{$\theta= \pi$}}] at ({cos ( (2+\aShift) *90)/2} , {sin((2+\aShift)*90)}){};
\draw node[circle,fill, inner sep =1pt, label= right:{\tiny{$\theta=\tfrac{3 \pi}{2}$}}] at ({cos ( (3+\aShift) *90)/2} , {sin((3+\aShift)*90)}){};
\end{scope}
\end{scope}

\end{scope}
\begin{scope}[scale=2]
\coordinate (C) at (1,0);
\tikzmath{ \m=4;\aShift=1/2; \maxx=1/3; \N=3;}
\draw (-1.5,1)--(2,1);
\draw (-1.5,-1)--(2,-1);
\begin{scope}[shift=(C)]

\draw[color=light-gray,dashed]   plot[smooth,domain=-1:1] ( {cos ((\x)*90)/2},{sin ((\x)*90)});
\draw[color=light-gray,dashed]   plot[smooth,domain=-1:1] ({-1.8+cos ((\x)*90)/2},{sin ((\x)*90)});

\foreach \i in {1,...,\N} {
\draw[->,color=light-gray,dashed,thick]   plot[smooth,domain=\maxx*(sqrt((\i-1)/\N)-.05):\maxx*sqrt(\i/\N)]  ( {-\m*tan (\x*90)* tan(\x*90) +cos ((\x+3+\aShift)*90)/2},{sin ((\x+3+\aShift)*90)});
\draw[->,color=light-gray,dashed,thick]   plot[smooth,domain=\maxx*(-sqrt(\i/\N)-.05):\maxx*(-sqrt((\i-1)/\N))]  ( {-\m*tan (\x*90)* tan(\x*90) +cos ((\x+3+\aShift)*90)/2},{sin ((\x+3+\aShift)*90)});
}

\foreach \i in {1,...,\N}{
\draw[black,->,thick]   plot[smooth,domain=\maxx*(sqrt((\i-1)/\N)-.05):\maxx*sqrt(\i/\N)]  ( {-\m*tan (\x*90)* tan(\x*90) +cos ((\x+1+\aShift)*90)/2},{sin ((\x+1+\aShift)*90)});
\draw[black,thick,->]    plot[smooth,domain=\maxx*(-sqrt(\i/\N)-.05):\maxx*(-sqrt((\i-1)/\N))]   ( {-\m*tan (\x*90)* tan(\x*90) +cos ((\x+1+\aShift)*90)/2},{sin ((\x+1+\aShift)*90)});
}

\draw[black]   plot[smooth,domain=1:3] ( {cos ((\x)*90)/2},{sin ((\x)*90)});
\draw[black]   plot[smooth,domain=1:3] ( {-1.8+cos ((\x)*90)/2},{sin ((\x)*90)});

\draw[shift=(C)] node[below] at (0,-1.2){\tiny{$x_1=0$}};
\draw[shift=(C)] node[below] at (-2,-1.2){\tiny{$x_1<0$}};
\draw[shift=(C)] node[below] at (-1,-1.5){{Tricomi}};
\draw node[circle,fill, inner sep =1pt, label= right:{\tiny{$\theta=0$}}] at ({cos ( (0+\aShift) *90)/2} , {sin((0+\aShift)*90)}){};
\draw node[circle,fill, inner sep =1pt, label= right:{\tiny{$\theta=\tfrac{ \pi}{2}$}}] at ({cos ( (1+\aShift) *90)/2} , {sin((1+\aShift)*90)}){};
\draw node[circle,fill, inner sep =1pt, label= right:{\tiny{$\theta= \pi$}}] at ({cos ( (2+\aShift) *90)/2} , {sin((2+\aShift)*90)}){};
\draw node[circle,fill, inner sep =1pt, label= right:{\tiny{$\theta=\tfrac{3 \pi}{2}$}}] at ({cos ( (3+\aShift) *90)/2} , {sin((3+\aShift)*90)}){};
\end{scope}

\end{scope}
\end{tikzpicture}

\caption{
A comparison of the Keldysh operator \eqref{eq:Keld0} and the Tricomi 
operator \eqref{eq:tric}. The figures show the 
cylinder $ \RR_{x_1} \times \mathbb S^1_\theta $ where
$ ( \xi_1 , \xi_2 ) = |\xi| ( \cos \theta, \sin \theta ) $ (this is the 
boundary of the fiber compactified cotangent bundle $ \overline{ T}^* \RR^n $
-- see \cite[\S E.1.3]{dizzy} -- with the $ x_2 $ variable omitted).
The characteristic varieties, $ x_1 \cos^2 \theta + \sin^2 \theta = 0 $
and $ \cos^2 \theta + x_1 \sin^2 \theta = 0 $, respectively, are shown
with the direction of the Hamiltonian flow indicated. 
In the the Keldysh case,
the two radial Lagrangians, $ \Lambda_\pm $, correspond to $ \theta = \pi $ and $ \theta = 0 $ respectively.
} \label{f:radial1}
\end{figure}
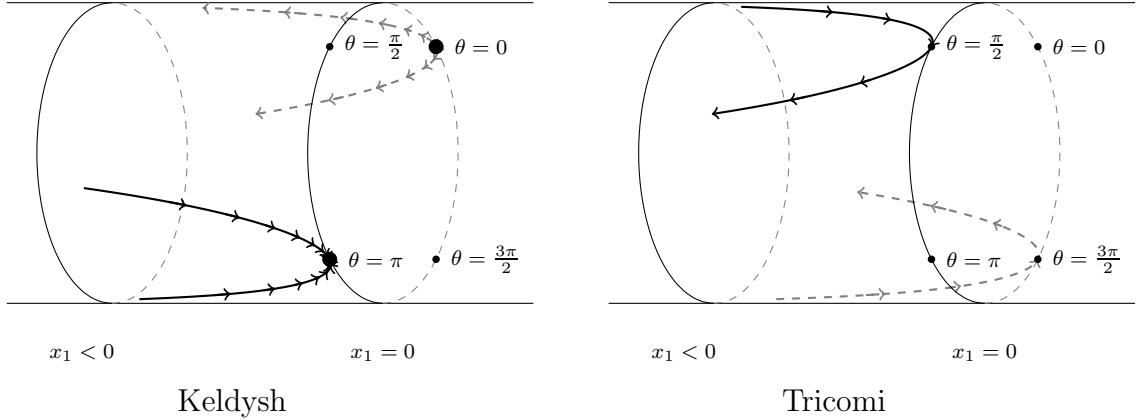

\noindent
2. The result is false for the Tricomi operator 
\begin{equation} 
\label{eq:tric} P := D_{x_1}^2 + x_1 D_{x_2}^2 . \end{equation}
This can be seen using results about propagation of analytic singularities (unlike \eqref{eq:Keld} this operator can be microlocally conjugated to 
$ D_{y_1 }$ -- see Figure \ref{f:radial1}) 
but is also easily demonstrated by the following example: 
\begin{equation}
\label{eq:Airy}
 u ( x ) := \int_{0}^\infty Ai ( \tau^{4/3 } x_1 ) e^{ i \tau^2 x_2} 
e^{ - \tau } d \tau ,  \ \ P u =  0 , \ \ u \in \CI ( \RR^2 ) .
\end{equation}
Here, $ Ai $ is the Airy function 
which satisfies
\begin{gather*}  Ai'' ( t ) + t Ai ( t ) = 0 , \ \ 
| \partial^{\ell}_t Ai ( t ) | \leq C_\ell \langle t \rangle^{\frac{\ell}{2}-\frac14} , \ \ t \in \RR, \ \ell \in \mathbb N, \ \ Ai ( 0 ) > 0 .
\end{gather*}
We then have
\[  D_{x_2}^k u ( 0) = Ai (0) \int_0^\infty \tau^{ 2k } e^{ - \tau } d\tau =  Ai ( 0 ) (2k)! 
\]
and $ u $ is not analytic at $ 0 $. 

\noindent
3. Results similar to \eqref{eq:t1} have been obtained in the setting of other operators. In addition to the works \cite{bc},\cite{bcz} cited
above, we mention the work of Baouendi--Sj\"ostrand \cite{BaS} who considered a class
of Fuchsian operators generalizing 
\begin{equation}
\label{eq:Fuch}   P = |x|^2 \Delta + \mu \langle x ,D_x \rangle + \lambda 
\end{equation}
In the case of \eqref{eq:Fuch},
\eqref{eq:t1} holds for any $ \lambda, \mu \in \CC $ and 
\cite{BaS} established \eqref{eq:t1} for more general operators satisfying
appropriate conditions.

\noindent
4. The operators \eqref{eq:Keld}, \eqref{eq:tric} and \eqref{eq:Fuch} 
are not $ \CI $ hypoelliptic, that is, $ P u \in \CI \not \Rightarrow 
u \in \CI $. The study of operators which are $ \CI $ hypoelliptic 
but not analytic hypoelliptic has a long tradition 
with a simple example \cite[\S 8.6, Example 2]{H1} given by  
\[ P = D_{x_1}^2 + x_1^2 D_{x_2}^2 + D_{x_3}^3 . \]
For more complicated cases, references, and connections to several complex variables, see Christ \cite{christ} and for some recent progress and additional references, 
Bove--Mughetti \cite{bomu}.

\subsection{A microlocal result}
\label{s:micro}

We make the following general assumptions. Let $ P$ be a differential operator of order $ m  $ with analytic coefficients:
\begin{equation}
\label{eq:defP}  P := \sum_{ |\alpha| \leq m } a_\alpha ( x ) D_x^\alpha, \ \ 
a_\alpha \in C^\omega ( U ) , \ \ p ( x , \xi ) := 
\sum_{ |\alpha | = m } a_\alpha ( x ) \xi^\alpha, \end{equation}
where $ U $ is an open neighbourhood of $ x_0 \in \RR^n $. We make the following assumptions valid in a conic neighbourhood of 
$ (x_0, \xi_0 )  \in T^* \RR^n \setminus 0 $: $ p $ is {\em real valued} and 
there exists a conic Lagrangian submanifold $ \Lambda $, such that
\begin{equation}
\label{eq:hypP} 
\begin{gathered}
( x_0 , \xi_0 ) \in \Lambda \subset p^{-1} ( 0 ) , \ \ \ 
dp|_{\Lambda } \neq 0 , \ \ \ H_p |_{\Lambda } \parallel \xi \cdot \partial_\xi |_{\Lambda } . 
\end{gathered}
\end{equation}
Here $ \parallel $ means that the two vector fields are {\em positively} proportional, 
that is the Lagrangian is {\em radial} (the positivity assumptions can be achieved by multiplying $ P $ by $ \pm 1 $). 
Except for the analyticity assumption in \eqref{eq:defP} these are the assumptions made in Haber \cite{hab} and Haber--Vasy \cite{hv}.

Theorem \ref{t:1} follows from the following microlocal result.
We denote by $ \WF $ the $ \CI $-wave front set 
and by $ \WFa $ the analytic wave front set -- 
see \cite[\S 8.1]{H1} and \cite[\S 8.5,9.3]{H1}, respectively.

\begin{theo}
\label{t:2}
Suppose that $ P $ and $ ( x_0 , \xi_0 ) \in T^* \RR^n \setminus 0 $
satisfy  the assumptions \eqref{eq:defP} and
\eqref{eq:hypP}. Then for $u\in \mathscr{D}'(\mathbb{R}^n)$,
\begin{equation}
\label{eq:t2}
 (x_0,\xi_0)\notin \WF(u) , \ \ (  x_0 , \xi_0 ) \notin 
\WFa ( Pu ) 
\ \Longrightarrow \ ( x_0 , \xi_0 ) \notin \WFa ( u ) . 
\end{equation}
\end{theo}


The proof is based on the theory of 
microlocal symbolic weights developed by 
Galkowski--Zworski \cite{gaz1} and based on the work 
of Sj\"ostrand -- see \cite[\S 2]{Sj96}
(and also \cite{HS} and \cite[\S 3.5]{M}).
With this theory in place we can use escape functions, $ G $, 
$ H_p G \geq 0 $, 
which are logarithmically bounded in $ \xi $ (hence the $ \CI $  wave front 
set  assumption on 
$ u $ allows the use of such weights) and which tend to $ \langle \xi \rangle$ in a neighbourhood of $ ( x_0 , \xi_0 ) $.
The normal form for $ p $ constructed in \cite{hab} (following much
earlier work of Guillemin--Schaeffer \cite{gus} which was based in turn on 
Sternberg's linearization theorem \cite{stern}) was helpful 
in the construction of the specific weights needed here. We indicate 
the method of the proof in \S \ref{s:idea}.

\begin{proof}[Proof of Theorem \ref{t:1}]
Under the assumptions of Theorem \ref{t:1} the characteristic set 
of $ P $ over $ x_1 = 0 $ is given by (in $ T^* \RR^n \setminus 0 $)
\[ p^{-1} ( 0 ) \cap \{ x_1 = 0 \} = \{ ( 0 , x_2 , \xi_1 , 0 ) :
 \xi_1 \in \RR \setminus 0 ; x_2 \in {\rm{neigh}}_{\RR^{n-1}} ( 0) \} = 
 \Lambda_+ \sqcup \Lambda_- , \]
 where $ \pm \xi_1>0 $ on $ \Lambda_\pm $. These two components are 
 Lagrangian and conic and $ H_p |_{\Lambda_\pm } = 
- \xi_1^2 \partial_{\xi_1} |_{\Lambda_\pm } $ is radial. 
Since  $ 
Pu \in C^\omega ( U ) $ we have $ \WFa ( P u ) \cap \{ x\in U : x_1 = 0 \} = \emptyset $ and hence Theorem \ref{t:2} shows that 
$ \WFa ( u ) \cap \Lambda_\pm = \emptyset $. On the other hand,
(\cite[Theorem 8.6.1]{H1}), $ \WFa ( u ) \cap \{ x_1 = 0 \} 
 \subset p^{-1} ( 0 ) \cap \{ x_1 = 0 \} = \Lambda_+ \sqcup \Lambda_- $.
Hence $ \WFa ( u ) \cap \{ x_1 = 0 \} = \emptyset$ and, since
$ {\rm{singsupp}_a} \, u = \pi \WFa ( u ) $, $ u $ is analytic near 
$ x_1 = 0 $.
\end{proof} 

\subsection{A proof in a special case}
\label{s:idea}
To indicate the ideas behind the proof we consider $ P $ given by 
\[  P = x_1 D_{x_1}^2 + D_{x_2}^2 + a D_{x_1} ,  \ \  a \in \CC , \]
 and a very special $ u $:
\begin{equation}
\label{eq:spass} 
 u = e^{i \tau x_2 } v ( x_1 ) , \ \ 
v \in \mathscr S ( \RR ) , \ \ 
Pu  = e^{ i \tau x_2} f ( x_1),    \ \ e^{  | \xi_1 | }
\widehat f \in L^2 ( \RR ) . 
\end{equation}
This assumption is a stronger version of the assumption that $ f $ is 
analytic. We consider a family of smooth functions $ G_\epsilon ( \xi_1 )  $ satisfying
\begin{equation}
\label{eq:Geps0}  0 \leq G_\epsilon ( \xi_1 ) \leq 
\min ( \tfrac 1 \epsilon \log ( 1 + | \xi_1| ) , | \xi_1 | ) \end{equation}
In view of 
\eqref{eq:spass},  
\[ \|  v_{\e}  \|_{ L^2 ( \RR ) } \leq 
C_\epsilon , \ \ \|   f_{\e}  \|_{ L^2 ( \RR ) } \leq C_0 \, 
\ \ v_{\e} := e^{ G_\epsilon  ( D_x ) } v, \ \
f_{\e} := e^{  G_\epsilon  ( D_x ) } f.  
\]
where $ C_0 $ is independent of $ \epsilon $. We then consider
\[ P_{\e} := e^{ G_\epsilon ( D_x ) } ( x_1 D_{x_1}^2 + a D_{x_1} + \tau^2 )
e^{ -  G_\epsilon ( D_x ) } = 
x_1 D_{x_1}^2 + i  G'_\epsilon ( D_{x_1} )D_{x_1}^2  + a D_{x_1} 
+ \tau^2  . \]
We have $ P_{\e} v_{\e} = f_{\e} $, 
and
\[ \begin{split}
 \Im \langle P_{\e} v_{\e} , v_{\e} \rangle_{ L^2 ( \RR ) }  & = 
 \langle   G'_{\epsilon}  ( D_{x_1} ) D_{x_1}^2 v_{\e} , v_{\e} \rangle_{L^2 ( \RR)}
+ \langle ( \Im a + 1 ) D_{x_1 } v_{\e} , v_{\e} \rangle_{L^2 ( \RR)}  \\
& =  \langle ( \xi_1^2 G'_\epsilon  ( \xi_1 ) + ( \Im a + 1 ) \xi_1 ) \widehat v_{\e} , \widehat 
v_{\e} \rangle_{ L^2 ( \RR_{\xi_1}  ) },
\end{split} 
\]
where we took $ d\xi_1/ ( 2 \pi ) $ as the measure on $ L^2 ( \RR_{\xi_1} ) $. 
Let $ \chi \in \CI ( \RR ; [0,1] ) $ satisfy $ \chi|_{ t \leq 1 } = 1 $,
$ \chi|_{ t \geq 2 } = 0 $ and $ \chi' \leq 0 $. We define 
\[ G_\epsilon ( \xi_1 ) = ( 1 - \chi( \xi_1 ) ) \int_0^{\xi_1 } 
( \chi ( \epsilon t ) + ( 1 - \chi ( \epsilon t )) ( \epsilon t )^{-1} )
dt , \]
which satisfies \eqref{eq:Geps0} and $ G'_\epsilon \geq 0 $. Moreover, for
$ \xi_1 \geq M \geq 2 $ and $ \epsilon < 1/M $, 
\[  \xi_1^2  G_\epsilon' ( \xi_1 ) \geq \xi_1^2 \chi( \epsilon \xi_1 ) + 
\e^{-1} \xi_1 ( 1 - \chi ( \epsilon \xi_1 ) ) \geq M \xi_1 . \]
Hence, by taking $ M = \max ( -\Im a + 1, 2 ) $, and $ \e < 1/M $, 
\[ \begin{split}   \| f_{\e} \| \| \widehat v_{\e} \| & \geq 
\Im \langle P_{\e} v_{\e} , v_{\e} \rangle  = 
\langle ( \xi_1^2 G_\epsilon' ( \xi_1 ) + (\Im a + 1 ) 
\xi_1 ) \widehat v_{\e} , \widehat v_{\e} \rangle \\
& \geq 
\| \widehat v_{\e} \|^2 - \| ( 1 +|\xi_1|(|\Im a|+1)) \widehat v_{\e} |_{\xi_1 \leq M } \| \|
\widehat v_{\e} \| \geq \| \widehat v_{\e} \|^2 
- C_1  \| \widehat v_{\e} \|, 
\end{split} \]
where $ C_1 := ( |\Im a | + 1)e^M \| v \|_{ H^1 } $ is independent of $ \e $.
This implies that
\[  \| \widehat v_{\e} \| \leq \|f_\e \| + C_1 \leq C_0 + C_1  . \]
Letting $ \epsilon \to 0 $ gives
$ \| e^{ \xi_1 } \widehat v |_{\xi_1 \geq 0 } \| \leq C $. A similar argument
applies to $ \xi_1 \leq 0  $ which shows that
\[   e^{ |\xi_1|}\widehat v \in L^2 , \]
and consequently that  $ u ( x ) = e^{ i x_2 \tau } v(x_1 ) $ is analytic. 

In the actual proof, the Fourier transform is replaced by the FBI transform
\eqref{eq:FBI} and its deformation \eqref{eq:TLa}
 defined using a suitably chosen $ G_\epsilon $ satisfying \eqref{eq:Geps0} (see Lemma \ref{l:G}
which is the heart of the argument). One difficulty not present in the simple one dimensional case is the localization in other variables.
It is here that the $ \CI $ normal forms of \cite{stern},\cite{gus} and \cite{hab} are particularly useful. It is essential that no analyticity is needed in the construction of $ G_\epsilon $.

\subsection{Applications to scattering theory}
\label{s:soma}

As already indicated in \cite{zulu} analyticity of smooth solution to the Vasy operator \eqref{eq:Plag} implies analyticity of resonant
states and of their radiation patterns. We review this here and, in Theorem \ref{t:3}, present a slightly stronger result. 

For a detailed presentation of scattering on asymptotically hyperbolic manifolds we refer to \cite[Chapter 5]{dizzy}. To state Theorem \ref{t:3}, let
$  \overline M$ be a compact  $ n+1$ dimensional manifold with boundary 
$ \partial M \neq \emptyset$ and let $ M :=  \overline  M  \setminus \partial M $. We assume that $ \overline M $ is a {\em real analytic} manifold near 
$ \partial M $.  
A metric $ g $ on $ M $ is called 
{\em asymptotically hyperbolic} and {\em analytic near infinity} 
if there exist functions $ y' \in \CI ( \overline M ; \partial M )$ and
 $ y_1  \in {\CI} (  \overline M ; (0, 2)  ) $, 
$ y_1|_{\partial M } = 0 $, $ dy_1|_{\partial 
M } \neq 0 $, such that 
\begin{equation}
\label{eq:coords}  \overline M \supset y_1^{-1} ( [0 , 1 ) ) \ni m \mapsto ( y_1( m), y'( m ) ) 
\in [ 0 , 1 ) \times \partial M \end{equation}
is a real analytic diffeomorphism, 
and near $ \partial M $ the metric has the form,
\begin{equation}
\label{eq:gash}   g|_{ y_1 \leq \epsilon }  = \frac{ dy_1^2 + h ( y_1 )  }{ y_1^2 } ,  \end{equation}
where 
 $ [ 0, 1 ) \ni t \mapsto h ( t ) $, 
is an analytic family of real analytic Riemannian metrics on $ \partial M $.

 Let 
 \[  R_g ( \lambda ) = ( - \Delta_g - \lambda^2 - (n/2)^2 )^{-1} :
L^2( M, d \vol_g ) \to H^2 ( M, d \vol_g ) , \ \  \Im \lambda > 0 . \]
Mazzeo--Melrose \cite{mm} and Guillarmou
\cite{g} proved that 
\begin{equation}
\label{eq:mgm}
\begin{gathered}
R_g ( \lambda ) : \CIc ( M ) \to \CI ( M ) , 
\end{gathered}
\end{equation}
continues to a meromorphic family of operators for $ \lambda \in \CC \setminus i( - {\textstyle{\frac12}}  - \mathbf N) $. In addition,
Guillarmou \cite{g}  showed that if the metric is {\em even}, that is,
\begin{equation}
\label{eq:gash1}   g|_{ y_1 \leq \epsilon }  = \frac{ dy_1^2 + h ( y_1^2 )  }{ y_1^2 } ,  \end{equation}
(see \cite[Theorem 5.6]{dizzy} for an invariant formulation), then 
$ R_g ( \lambda ) $ is meromorphic in $ \CC $. In particular, 
for $ \lambda \neq 0  $ we have the following
Laurent expansion
\[  R_g ( \zeta ) = \sum_{ j=1}^{J ( \lambda ) } \frac{ ( - \Delta_g - \lambda^2 - 
(n/2)^2 )^{ j-1} \Pi ( \lambda ) }{ ( \zeta^2 
- \lambda^2)^j} + A ( \zeta, \lambda ) , \
\ \ \Pi (\lambda ) := \frac{1}{ 2 \pi i } \oint_\lambda
R_g ( \zeta ) 2 \zeta d \zeta ,
\]
where $ \zeta \mapsto A ( \zeta, \lambda ) $ is holomorphic near $ \lambda $. 
For $ \lambda = 0 $ we have a Laurent expansions in powers of 
$ \zeta^{-j} $. 

The operator $ \Pi ( \lambda ) $ has finite rank and its range consists of {\em generalized resonant states}. We then have
\begin{theo}
\label{t:3}
Suppose that $ ( M , g )  $ is an even  asymptotically hyperbolic manifold 
(in the sense of \eqref{eq:gash1}) 
analytic near conformal infinity $ \partial M $.
Then for $ \lambda \in \CC \setminus 0 $, 
\begin{equation}
\label{eq:mgm2}
\begin{gathered}
u \in \Pi ( \lambda ) \CIc ( M )  
 \ \ \Longrightarrow \ \  
u =  y_1^{ - i \lambda + \frac{n}{2} } F,  \ \ F |_{\partial M } 
\in C^\omega ( \partial M ) .
\end{gathered} \end{equation}
Moreover, in coordinates of \eqref{eq:gash1}, $ F ( y ) = f ( y_1^2, y' ) $, $ y' \in \partial M $ where $ f \in C^\omega ( (-\delta, \delta) 
\times \partial M ) $.
\end{theo}
\begin{proof}
The metric \eqref{eq:gash} (in the coordinates valid near the boundary) gives the following Laplace operator:
\begin{gather}
\label{eq:Deltagg} 
\begin{gathered} - \Delta_g = ( y_1 D_{y_1} )^2 + 
i ( n  + y_1 \gamma_0 ( y_1^2, y') ) y_1
D_{y_1} - y_1^2 \Delta_{h (y_1 ) } , \\
\gamma_0 ( t, y') := - {\textstyle{\frac12}} \partial_t \bar h ( t ) / \bar h ( t ) , \ \
\bar h ( t ) := \det h ( t ) , \ \ D := \textstyle{\frac 1 i } \partial .
\end{gathered}
\end{gather} 
Following Vasy \cite{vas} we change the variables to $ x_1 = y_1^2 $,
$ x' = y'$ so that
\begin{equation}
\label{eq:firstconj} y_1^{ i \lambda - \frac n 2 }   ( - \Delta_g -
\lambda^2 - ({\textstyle{\frac n 2}})^2  ) y_1^{-i \lambda + \frac n 2 } =
x_1 P ( \lambda ) ,  \end{equation}
where, near $ \partial M $, $ P ( \lambda ) $ is given by 
\eqref{eq:Plag}. This operator is considered on $ X := 
(( - \delta, 0 ]_{x_1} \times \partial M) \sqcup M $. 
 The key fact is that $ P ( \lambda) $ is 
a Fredholm family operators on suitable spaces, $ P ( \lambda )^{-1} $ 
is meromorphic and its poles can be studied using microlocal methods -- 
see \cite{vas}, \cite[Chapter 5]{dizzy} and also \cite[\S 2]{V4D} for a short self-contained presentation. 

From meromorphy of $ P ( \lambda )^{-1} $ we obtain meromorphy of
\eqref{eq:mgm} using \eqref{eq:firstconj}:
\begin{equation}
  \label{e:ah-merinv-constructed}
R_g (\lambda)f:=y_1^{{n\over 2}-{i\lambda}}\big(P(\lambda)^{-1}y_1^{{i\lambda}-{n+2\over 2}}f\big)\big|_M\ \in\ C^\infty(M).
\end{equation}
Here we make $y_1^{{i\lambda}-{n+2\over 2}}f$ into an element of $\CIc(X)$
by extending it by zero outside of $M$. 
Near any $ \lambda $, $  P( \zeta )^{-1}  = \sum_{k=1}^{K ( \lambda )} { Q_j ( \lambda )}{ (\zeta - \lambda )^{-j } } 
+ Q_0 ( \zeta, \lambda )$, 
with $ Q_j (\lambda ) $ operators of finite rank and $ \zeta \mapsto
Q_0 ( \zeta, \lambda ) $ is analytic near $ \lambda $. We then have 
\[  \Pi ( \lambda ) = \tfrac{1}{ 2 \lambda } y_1^{{n\over 2}-{i\lambda}} 
Q_1 ( \lambda ) y_1^{{i\lambda}-{n+2\over 2}}  . \]
Hence, the claim about the range of $ \Pi ( \lambda ) $ follows from 
analyticity of functions in the range of $ Q_1 ( \lambda ) $. This
follows from Theorem \ref{t:1}. In fact, 
$ P ( \zeta ) = P ( \lambda ) + ( \zeta - \lambda ) V $, 
$ V := - 4 D_{x_1} + i \gamma ( x ) $, and hence
\[  P ( \lambda ) Q_k ( \lambda ) = - V Q_{k+1} ( \lambda) , 
\ \ Q_{ K+1} ( \lambda ) := 0 . \]
Since we already know that the ranges of $ Q_k $'s are in $ \CI $ 
(see \cite[(5.6.10)]{dizzy}) 
we inductively conclude that the ranges are in $ C^\omega$. 
\end{proof}

\noindent
{\bf Remark.} Vasy's adaptation of Melrose's radial estimates \cite{mel} shows that to conclude that 
$ u \in \CI $ when $ P ( \lambda ) u \in \CI $ (see \eqref{eq:Plag}), we only need to assume that $ u \in H^{s+1} $ near $ m_0 $, where
$  s + {\textstyle{\frac12}} > - \Im \lambda$, 
see \cite[\S 4, Remark 3]{V4D}.

\section{Preliminaries on FBI transforms and their deformations}
\label{s:FBI}

We will use the FBI transform defined in \cite{gaz1} in its $ \RR^n $
(rather than $ \TT^n $) version. Since the weights we use will be compactly supported in $ x $ the same theory applies. The constructions there are inspired by the works of 
Boutet de Monvel--Sj\"ostrand \cite{BS}, 
Boutet de Monvel--Guillemin \cite{BG}, Helffer--Sj\"ostrand \cite{HS}
and Sj\"ostrand \cite{Sj96}. An alternative approach to using the classes of weights we need here was developed independently and in greater generality by 
Guedes Bonthonneau--J\'ez\'equel \cite{Guj}.

\subsection{Deformed FBI transforms}\label{s:dFBI}
We define
\begin{equation} 
\label{eq:FBI}
T u ( x, \xi )  :=  h^{-\frac{3n}{4}}\int_{\RR^n} 
e^{ \frac i h ( \langle x-y,\xi\rangle +\tfrac{i}{2}\langle \xi\rangle (x-y)^2)} 
\langle \xi\rangle^{\frac{n}{4}}u(y)dy, \ \ u\in \CIc (\RR^n ),
\end{equation}
recalling that
the left inverse of $ T $ is given by 
\begin{equation}
\label{eq:iFBI}
S v ( y ) =\frac{ 2^{\frac n 2}h^{-\frac{3n}{4}} }{(2\pi )^{\frac{3n}{2}}} \int_{\RR{2n} } e^{-\frac{i}{h}
( \langle x-y,\xi\rangle - \tfrac{i}{2}\langle \xi\rangle (x-y)^2)}\langle \xi\rangle^{\frac{n}{4}
} (1+\tfrac{i}{2} \langle x-y, {\xi}/{\langle \xi\rangle}\rangle)
  v(x,\xi)dxd\xi, 
\end{equation}
see \cite[Proposition 2.2]{gaz1}. 

The first fact we need is the characterization of Sobolev spaces
and of the $ \CI $ wave front set using the FBI transform \eqref{eq:FBI}.
To formulate it we use semiclassical Sobolev spaces $ H_h^s $ 
(see for instance \cite[\S 7.1]{zw} or \cite[Definition E.18]{dizzy}) 
but we should in general think of $ h $ as being fixed.

\begin{prop}
\label{p:FBI2H}
There exists a constant $ C $ such that for $ u \in \mathscr S'(\RR^n)  $, 
\begin{equation}
\label{eq:Hs2xi}    
\| u \|_{ H^s_h } \leq C \| \langle \xi \rangle^s T u \|_{ L^2 ( T^* \RR^{n} ) } \leq C^2  \| u \|_{ H^s_h} .
\end{equation}
Moreover, 
\begin{equation*}
( x_0, \xi_0 ) \notin \WF ( u ) \, \Leftrightarrow
\, \left\{ \begin{array}{l} \exists \, \chi \in S^0 ( T^*\RR^n ) , \
\text{$ \chi \equiv 1 $ in a conic neighbourhood of $ ( x_0, \xi_0 ) $,}  \\
\forall \, N \ \exists \, C_N  \ \ \| \langle \xi \rangle^N \chi T u \|_{ L^2 ( T^* \RR^n ) } \leq C_N .
\end{array} \right.
\end{equation*}
\end{prop}
\begin{proof}
This follows from the characterization of the $ H^s $ based wave front 
sets in G\'erard \cite{ger} as stated in \cite[Theorem 1.2]{De}. Since the arguments are similar to the more involved analytic case presented in 
Proposition \ref{p:awf} we omit the details.
\end{proof}

As in \cite[\S 2]{Sj96} and \cite[\S 3]{gaz1} we introduce a geometric deformation of 
$ \RR^{2n} $, $ \Lambda = \Lambda_G$:
\begin{equation}
\label{e:LambdaDef}
\begin{gathered}
\Lambda:=\{(x-iG_\xi(x,\xi),\xi+iG_x(x,\xi))\mid (x,\xi)\in \RR^{2n} \}\subset \CC^{2n}, \\
\supp G \subset K \times \RR^n , \ \ K \Subset \RR^n , \\ 
\sup_{ |\alpha| + |\beta| \leq 2 } \langle \xi\rangle^{-1+|\beta|} | \partial_x^\alpha \partial_\xi ^\beta G(x,\xi) | \leq \epsilon_0, \ \ 
 | \partial_x^\alpha \partial_\xi ^\beta G(x,\xi) | \leq C_{\alpha \beta} \langle \xi\rangle^{1-|\beta|},
\end{gathered}
\end{equation}
where $ \epsilon_0 $ is small and fixed (so that the constructions below 
remain valid as in \cite{gaz1}).  For convenience, we change here
the convention from \cite{gaz1}: it amounts to to replacing $ G $ by 
$ - G $ everywhere.

This provides us with the following new objects: the 
deformed FBI transform (see \cite[\S 4]{gaz1}),
\begin{equation}
\label{eq:TLa}
\begin{gathered}
T_\Lambda u ( x, \xi ) :=  T u ( x - i G_\xi ( x, \xi ) , \xi + i G_x (x, \xi)) ,  \ \  u \in \mathscr B_\delta , \\
\mathscr B_\delta := \{ u \in \mathscr S ( \RR^n )  :  \int_{ \RR^n }
|\widehat U ( \xi)|^2 e^{ 4 \delta |\xi| } d\xi  < \infty \}, 
\end{gathered}
\end{equation}
the the spaces $ H_\Lambda^s $, defined as in \cite[\S 4]{gaz1},
\begin{equation}
\label{eq:HLa}
H_\Lambda^s := \overline{\mathscr B_{\delta_0}}^{ \| \bullet \|_{H^s_\Lambda } } , \ \ 
\| u \|_{ H^s_\Lambda }^2 := \int_{\Lambda} \langle \Re \alpha_\xi \rangle^{2s} 
| T_\Lambda u ( \alpha )|^2 e^{-2 H ( \alpha ) /h } d \alpha , 
\end{equation}
and the orthogonal projector
\[ \Pi_\Lambda : L_\Lambda := L^2 ( \Lambda , e^{-2 H ( \alpha ) /h } d \alpha) \to T_\Lambda H_\Lambda , \ \ \ H_\Lambda := H_\Lambda^0 , \]
described asymptotically (as $ h \to 0 $ and as $ \xi \to \infty $) in 
\cite[\S 5]{gaz1}. The weight $ H $ appears naturally in this subject and is given by \cite[(3.3),(3.4)]{gaz1} i.e. $H(x,\xi)=\xi\cdot G_\xi(x,\xi)-G(x,\xi)$. The deformed FBI transform $ T_\Lambda $ has an exact left inverse $ S_\Lambda $ obtained by 
deforming $ S$ in \eqref{eq:iFBI}.

We now prove a slightly modified version of \cite[Proposition 6.2]{gaz1}:
\begin{prop}
\label{p:6.2}
Suppose that $ P =  \sum_{|\alpha| \leq m } a_\alpha D^\alpha $ is a
differential operator with $ a_\alpha \in \CIc ( \RR^n ) $ satisfying,
\[  a_\alpha \in C^\omega ( U ) , \ \ K \Subset U , \]
for an open set $ U $ and $ K$ as in \eqref{e:LambdaDef}. Then 
\[  \Pi_\Lambda T_\Lambda h^m P S_\Lambda = \Pi_\Lambda b_P \Pi_\Lambda 
+ \mathcal O ( h^\infty)_{ H^{-N}_\Lambda  \to H^N_\Lambda } , \]
where 
\begin{equation}
\label{eq:p6.2} 
\begin{gathered}
 b_P ( x, \xi )  \sim \sum_{j=0}^\infty h^j b_j ( x , \xi) , 
\ \ b_j \in S^{m-j} ( \RR^{2n} ) , \\ b_0 = p|_\Lambda := 
p (  x - i G_\xi ( x, \xi ) , \xi + i G_x ( x, \xi)) . 
\end{gathered}
\end{equation}
\end{prop}
We remark that the expansion remains valid when $ h $ is fixed. We can use 
smallness of $ h $ to dominate the lower order terms and then keep it fixed. 
\begin{proof}
The result follows from the analogue of \cite[Lemma 6.1]{gaz1} where
the operator $ T_\Lambda h^m P S_\Lambda $ is described in the case
where the coefficients of $ P $ are globally analytic. Here
we point out that the analyticity of the coefficients is only needed
in the neighbourhood $ U $ of $ K \Subset \RR^n $ such that in 
\eqref{e:LambdaDef} 
$ \supp G \subset K \times \RR^n $ and $ \epsilon_0 $ is small enough
depending on the size of the complex neighbourhood to which 
the coefficients extend holomorphically.

In fact, arguing as in the proof of \cite[Proposition 6.2]{gaz1} all we need is that for $ a \in \CIc ( \RR^n ) $ and $ a \in C^\omega ( U ) $, the Schwartz kernel of 
$ T_\Lambda M_{a } S_\Lambda $, $ M_a f ( x ) := a ( x ) f (x ) $, 
 is given by 
\begin{equation}
\label{eq:Kea}
\begin{gathered}K_a ( \alpha, \beta ) = c_0 h^{-n} e^{ \frac{i}{h} \Psi ( \alpha, 
\beta) } A ( \alpha , \beta ) + r ( \alpha, \beta ) , 
\ 
\ \ \alpha, \beta \in \Lambda = \Lambda_G ,\\
{r(\alpha,\beta)\text{ is the kernel of an operator }R=O(h^\infty) :  
{{H_{\Lambda}^{-N}\to H_{\Lambda}^N.}}}
  \end{gathered}
\end{equation}
The phase in \eqref{eq:Kea} is given by 
\begin{equation}
\label{eq:Psi}  \Psi (  \alpha, \beta ) = 
\frac{i}{2}\frac{(\alpha_\xi-\beta_\xi)^2}{\langle \alpha_\xi\rangle +\langle \beta_\xi\rangle}+\frac{i}{2}\frac{ \langle \beta_\xi\rangle \langle \alpha_\xi\rangle (\alpha_x-\beta_x)^2}{\langle \alpha_\xi\rangle +\langle \beta_\xi\rangle}+\frac{\langle \beta_\xi\rangle \alpha_\xi+\langle \alpha_\xi\rangle \beta_\xi}{\langle \alpha_\xi\rangle +\langle \beta_\xi\rangle}\cdot (\alpha_x-\beta_x),
\end{equation}
and the amplitude satisfies
\[ A \sim \sum_{j=0}^\infty h^j \langle \alpha_\xi \rangle^{-j} 
A_j,  \ \ \  A_0 ( \alpha , \alpha ) = a|_{\Lambda } ( \alpha)  , \]
and $ A_j $ are supported in a small conic neighbourhood of the diagonal in  $ \Lambda \times \Lambda $. 
We
note that if $ \epsilon_0 $ is small enough, $ a $ extends to some neighbourhood of $ K $ in $ \CC^n $ and hence $ a| _{\Lambda }
=  a ( x - i G_\xi ( x, \xi ) ) $ is well defined. 

To see \eqref{eq:Kea} we use the definitions of $ T_\Lambda $ and 
$ S_\Lambda $ to write
\begin{equation}\label{e:multiplyMe} K_a ( \alpha, \beta ) = 
c_n  \langle \beta_\xi \rangle^{\frac n 4}  \langle \alpha_\xi\rangle^{\frac{n}{4}} {h^{-\frac{3n}{2}}} \int e^{\frac{i}{h}(\varphi_G(\alpha,y) + \varphi_G^*(\beta,y))} a ( y ) \left( 1 + 
\langle \beta_x - y , \beta_\xi/\langle \beta_\xi \rangle 
\right) dy,\end{equation}
where
\begin{equation}
\label{eq:lost}
\begin{gathered}
\varphi_G ( \alpha, y ) = \Phi ( z , \zeta , y )|_{ z = \alpha_x, 
\zeta = \alpha_\xi } , \ \  
\varphi_G^* ( \alpha, y ) = - \bar \Phi ( z, \zeta , y )|_{ z = \alpha_x, 
\zeta = \alpha_\xi } , \\
\alpha_x = x - i G_\xi ( x, \xi) , \ \ \ 
\alpha_\xi = \xi + i G_x ( x, \xi ) , \\ 
\Phi ( z, \zeta , y ) = \langle z - y , \zeta \rangle + \tfrac i 2
\langle \zeta \rangle ( z - y )^2 , \ \ \
\bar \Phi ( z , \zeta, y ) := \overline{ \Phi ( \bar z , \bar \zeta , y ) }. 
\end{gathered}
\end{equation}

Let $V,V_1$ open such that $K\subset V_1\Subset V\Subset U$. We start by showing that the contribution to $K_a$ away from the diagonal is negligible. For that let $\chi\in C_c^\infty( \mathbb{R})$ with $\chi \equiv 1$ near $0$. Then for all $\delta>0$ small enough, the operator $R_1$ with kernel 
\begin{gather*}
R_1(\alpha,\beta)= K_a(\alpha,\beta)\tilde{\chi}_\delta(\alpha,\beta),\\
\tilde{\chi}_{\delta}(\alpha,\beta):=(1-\chi(\delta^{-1}|\alpha_x-\beta_x|))\left(1-\chi \Big(\frac{|\alpha_\xi-\beta_\xi|}{\delta 
{ \langle |\alpha_\xi-\beta_\xi| \rangle}}\Big)\right)
\end{gather*}
satisfies $R_1=O_{H_{\Lambda}^{-N}\to H_{\Lambda}^N}(h^\infty).$ This amounts to showing that the operator with kernel 
$
R_1(\alpha,\beta)e^{\frac{1}{h}(H(\beta)-H(\alpha))}\langle \alpha_\xi\rangle^N\langle \beta_\xi\rangle^N
$
is bounded on $L^2(\mathbb{R}^{2n})$ with $O(h^\infty)$ norm.

To see this, we first integrate by parts $K$ times in $y$, using that 
$$
|\partial_y\Psi|=|\beta_\xi-\alpha_\xi+i(\langle \alpha_\xi\rangle(y-\alpha_x)+\langle \beta_\xi\rangle(y-\beta_x))|\geq c {\left( 1 + |\alpha_\xi|+|\beta_\xi|  \right)}
$$
on $\supp\tilde{\chi}_\delta$. This reduces the analysis to the case of~\eqref{e:multiplyMe} with $a$ is replaced by $b(\cdot, \alpha,\beta)\in C^\omega(U)\cap C_c^\infty(\mathbb{R}^n)$ with 
$
|b|\leq h^K(\langle |\alpha_\xi|\rangle +\langle |\beta_\xi|\rangle)^{-K}.
$

Next, we choose $\psi\in C_c^\infty(\mathbb{R}^n;[0,1])$ with $\psi\equiv 1$ on $V$ and $\supp \psi\subset U$, and $\psi_1\in C_c^\infty(\mathbb{R}^n;[0,1])$ with $\psi_1\equiv 1$ on $V_1$ and $\supp \psi_1\subset V$. We then deform the contour
$$
y\mapsto y+i\epsilon\psi(y)\frac{\overline{\beta_\xi-\alpha_\xi}}{\langle |\beta_\xi-\alpha_\xi|\rangle}.
$$
This contour deformation is justified since $a\in C^\omega(U)$. The phase in  the integrand of~\eqref{e:multiplyMe} becomes
\begin{align*}
\Psi=&\langle \alpha_x-y,\alpha_\xi\rangle +\langle y-\beta_x,\beta_\xi\rangle +\frac{i\langle \alpha_\xi\rangle}{2}(\alpha_x-y)^2+\frac{i\langle \beta_\xi\rangle}{2}(\beta_x-y)^2\\
&+i\e\psi(y)\frac{|\beta_\xi-\alpha_\xi|^2}{\langle |\beta_\xi-\alpha_\xi|\rangle}+\frac{i\langle \alpha_\xi\rangle}{2}\Big[ 2\e\psi(y)\langle \alpha_x-y,\frac{\overline{\alpha_\xi-\beta_\xi}}{\langle |\beta_\xi-\alpha_\xi|\rangle}\rangle-\e^2\psi^2(y)\frac{|\beta_\xi-\alpha_\xi|^2}{\langle |\beta_\xi-\alpha_\xi|\rangle^2}\Big]\\
&\frac{i\langle \beta_\xi\rangle}{2}\Big[ 2\e\psi(y)\langle \beta_x-y,\frac{\overline{\alpha_\xi-\beta_\xi}}{\langle |\beta_\xi-\alpha_\xi|\rangle}\rangle-\e^2\psi^2(y)\frac{|\beta_\xi-\alpha_\xi|^2}{\langle |\beta_\xi-\alpha_\xi|\rangle^2}\Big]
\end{align*}
In particular, for $y\in V$, and $(\alpha,\beta)\in \supp\tilde{\chi}_\delta$, the integrand is bounded by 
$$
e^{-c(\langle \alpha_\xi\rangle+\langle \beta_\xi\rangle)\langle \alpha_x-\beta_x\rangle/h}
$$
which is negligible (even after multiplication by $e^{\frac{1}{h}(H(\beta)-H(\alpha))}\langle\alpha_\xi\rangle^N\langle \beta_\xi\rangle^N$).

For the integral over $y\notin V$, we consider three cases. First, if both $\Re \alpha_x\in K$ and $\Re \beta_x\in K$, then it is easy to see that the integrand is bounded by 
$$
e^{-c(\langle \alpha_\xi\rangle+\langle \beta_\xi\rangle)(\langle \alpha_x-\beta_x\rangle +|y|)/h}
$$
and hence produces a negligible contribution. Next, if $\Re \alpha_x\notin K$ and $\Re \beta_x\notin K$, then $H(\alpha)=H(\beta)=0$, $\alpha,\beta$ are real, and integration by parts in $y$ shows that the contribution is negligible.

Finally, we consider the case $\Re \alpha_x\in K$, $\Re \beta_x\notin K$, (the case $\Re \beta_x\in K$ and $\Re \alpha_x\notin K$ being similar). In this case, we have $H(\beta)=0$ and $\beta$ real. Since $y\notin V$, we have that the integrand is bounded by 
$
e^{-c\langle \alpha_\xi\rangle\langle \alpha_x-y\rangle/h}h^K\langle \beta_\xi\rangle^{-K}
$
and hence this term is also negligible.

Since $R$ is negligible, we may assume from now on that  
$$
|\alpha_x-\beta_x|\ll 1 \ \text{ and } \ 
|\alpha_\xi-\beta_\xi|\ll \langle|\alpha_\xi|\rangle+\langle |\beta_\xi|\rangle.$$
In particular, there are three cases: $\Re \alpha_x\in K$ and $\Re \beta_x\in V_1$, $\Re \beta_x\in K$ and $\Re \alpha_x\in V_1$, or $\Re\alpha_x\notin K$ and $\Re \beta_x\notin K$.

The first two cases are similar, so we consider only one of them. Since $\Re \alpha_x\in K$ and $\Re \beta_x\in V_1$, the contribution from  $y\notin V$ is negligible. Therefore, we may deform the contour to 
$$
y\mapsto y+ \psi(y)y_c(\alpha,\beta),\qquad y_c(\alpha,\beta)=\frac{i(\beta_\xi-\alpha_\xi)+\langle \alpha_\xi\rangle \alpha_x+\langle \beta_\xi\rangle \beta_x}{\langle \alpha_\xi\rangle+\langle \beta_\xi\rangle}.
$$
The proof in this case then follows from the method of complex stationary phase. 

When, both $\Re\alpha_x\notin K$ and $\Re \beta_x\notin K$, $\alpha=\Re \alpha$, $\beta=\Re \beta$, and $H(\alpha)=H(\beta)=0$. In order to handle this situation, we will Taylor expand $a(y)$ around $y=\alpha_x$. For that we first consider~\eqref{e:multiplyMe} with $a=O(|y-\alpha_x|^{2N})$.  In that case, we consider the integral
\begin{equation}
\label{e:K_N}
\begin{aligned}
&K_N(\alpha,\beta):=h^{-\frac{3n}{2}}\int e^{\frac{i}{h}(\langle \alpha_x-y,\alpha_\xi\rangle +\frac{i}{2}(\langle \alpha_\xi\rangle (\alpha_x-y)^2+\langle \beta_\xi\rangle (\beta_x-y)^2))}\\ &\qquad\qquad\qquad\qquad\qquad\qquad O(|y-\alpha_x|^{2N})\langle \alpha_\xi\rangle^{\frac{n}{4}}\langle \beta_\xi\rangle^{\frac{n}{4}}(1-\tilde{\chi}_\delta(\alpha,\beta))dy.
\end{aligned}
\end{equation}
Changing variables $y\mapsto y+\alpha_x$, 
\begin{align*}
|K_N(\alpha,\beta)|\leq&\int \langle \alpha_\xi\rangle^{\frac{n}{4}}\langle \beta_\xi\rangle^{\frac{n}{4}} \frac{h^{N-\frac{3n}{2}}}{\langle \alpha_\xi\rangle^N}e^{-\frac{\langle \beta_\xi\rangle}{2h}(\beta_x-\alpha_x-y)^2}(1-\tilde{\chi}_\delta)dy\\
&\qquad\qquad\leq C \frac{h^{N-n}}{(\langle \alpha_\xi\rangle+\langle \beta_\xi\rangle)^{N}}e^{-c\frac{\langle \alpha_\xi\rangle +\langle \beta_\xi\rangle}{h}(\alpha_x-\beta_x)^2}(1-\tilde{\chi}_\delta(\alpha,\beta)).
\end{align*}
Therefore, using the Schur test for boundedness, the operator $K_N$ with kernel $K_N(\alpha,\beta)$ satisfies
$$
K_N =O( h^{N-\frac{n}{2}}):H^{-N+\frac{n}{4}+0}_{\Lambda}\to H^{N-\frac{n}{4}-0}_\Lambda
$$
Now, observe that for any $N>0$, 
$$
a(y)=a_N(y)+O(|y-\alpha_x|^{2N})
$$
where $a_N(y)$ is a polynomial of order $2N-1$ in $(y-\alpha_x)$. In particular, 
$$
K_a(\alpha,\beta)=K_{a_N}(\alpha,\beta)+K_N(\alpha,\beta)
$$
Since $a_N$ is analytic and the integrand is exponentially decaying in $y$, we may deform the contour with $y\mapsto y+y_c(\alpha,\beta)$ in the integral forming the kernel of $K_{a_N}$ and apply complex stationary phase as in the case where $\Re \alpha_x\in K$ or $\Re \beta_x\in K$. This finishes the proof of the proposition after taking $N$ large enough.
\end{proof}

\subsection{Analytic wave front set}
\label{s:awf}

We now relate weighted estimates to analyticity. 
\begin{prop}
\label{p:awf}
Let $ T $ be the FBI transform defined in \eqref{eq:FBI} for some 
fixed $ h $, and let $ \psi \in S^1 ( T^* \RR^n  ) $ satisfy 
\begin{equation}
\label{eq:cond_psi}   \psi ( x, \xi )  \geq | \xi | /C , \ \ ( x, \xi ) \in
U \times \Gamma, \end{equation}
where $ U \subset \RR^n $ and $ \Gamma \subset \RR^n \setminus 0 $ is
an open cone. Then, for $ u \in {H^{-N}} ( \RR^n ) $, 
\begin{equation}
\label{eq:awf}  e^{\psi  } { \langle\xi\rangle^{-N}}T u \in L^2 ( T^* \RR^n ) \ \Longrightarrow \
\WFa ( u ) \cap (U \times \Gamma) = \emptyset . \end{equation}
Conversely, {suppose $u\in H^{-N} ( \RR^n ) $}, $ \Gamma_0 \subset \RR^n $ is a conic open set such that
$ \Gamma_0 \cap \mathbb S^{n-1} \Subset \Gamma \cap \mathbb S^{n-1} $, 
 $ U_0 \Subset U $.
Then for any $ \psi \in S^1 ( \RR^n \times \RR^n ) $ with 
$ \supp \psi \subset U_0 \times V_0 $,  
\begin{equation}
\label{eq:cawf}
\WFa ( u ) \cap (U \times \Gamma) = \emptyset 
\ \Longrightarrow \ \exists \, \theta > 0 \ \ \langle \xi\rangle^{-N}e^{\theta \psi  } T u \in L^2 ( T^* \RR^n ) . 
\end{equation} 
\end{prop}

\noindent
{\bf Remark:} Here we do not consider uniformity in $ h $ in the $ L^2 $ bounds. If we demanded that, than we would only need
$ \psi \in \CIc ( T^* \RR^n )$, $ \psi > 0 $ on $ U \times (\Gamma 
\cap \mathbb S^{n-1} ) $. 

The proof is based on the following
\begin{lemm}
\label{l:TS}
Let $ T $ and $ S $ be given by \eqref{eq:FBI} and \eqref{eq:iFBI}, 
respectively, with $ h $ fixed.
Suppose that $ \chi, \tilde \chi \in S^0 ( \RR^n \times \RR^n ) $ and 
$ \supp \chi, \supp  \chi_1 \subset K \times \RR^n $, $ K 
\Subset \RR^n $.
Then for any $ a > 0 $ there exists $ b > 0 $ such that
\begin{equation}
\label{eq:TS}
  \chi  e^{ b \langle \xi \rangle } 
T S  \chi_1 e^{ - a \langle \xi \rangle }  = 
\mathcal O_N ( 1 ) : L^2 ( \RR^{2n}) \to H^N ( \RR^{2n} ) , \end{equation}
for any $ N $.  

If in addition $ \chi_1 \equiv 1 $ on a a conic neighbourhood of the support of $ \chi $, then there exists $ b > 0 $ such that
\begin{equation}
\label{eq:TS1}   \chi  e^{ b \langle \xi \rangle } 
T S ( 1 -   \chi_1) {\langle \xi\rangle^M}   = 
\mathcal O_{N,M} ( 1 ) : L^2 ( \RR^{2n}) \to H^N ( \RR^{2n} ) , \end{equation}
for any $ N $.  
\end{lemm}
\begin{proof}
We analyse the Schwartz kernel of the operator in \eqref{eq:TS}, 
$ K ( x, \xi , y , \eta ) $. As in the proofs of 
\cite[Lemma 2.1, Proposition 4.5]{gaz1} (the phase of resulting operator
can be computed by completion of squares and is given by \cite[(4.10)]{gaz1} 
with $ \Lambda = T^* \RR^n $)
we see that
\begin{equation}
\begin{gathered} 
\label{eq:DalK} | ( h D)_{x,\xi}^\alpha K ( x, \xi , y , \eta ) |
\leq C_\alpha  
e^{  b \langle \xi \rangle - 
a \langle \eta \rangle - \psi ( x, \xi , y , \eta )  } , \\
\psi := 
c (\langle \xi \rangle + \langle \eta\rangle )^{-1} \left(  | \xi - \eta|^2 
 +  \langle \xi \rangle \langle \eta \rangle | x - y |^2 \right).
 \end{gathered}
 \end{equation}

We have 
\[ b < \tfrac 1 8 \min ( a, c) \ \Rightarrow \  b \langle \xi \rangle - a \langle \eta \rangle 
- c ( \langle \xi \rangle + \langle \eta \rangle )^{-1}  | \xi - \eta|^2 
\leq 
- \tfrac12 ( b \langle \xi \rangle + a \langle \eta \rangle ) , 
\]
if $ b $ is sufficiently small. (By taking $ b < a/8 $ we can assume that
$ |\eta|  \leq | \xi| /2 $. But then 
$ | \xi - \eta | \geq  \frac12  |\xi | $ and $ \langle \xi \rangle 
+ \langle \eta \rangle \leq 2\langle \eta \rangle $.) This proves
\eqref{eq:TS} as we can use the Schur criterion.

{To see \eqref{eq:TS1} we note that we can now assume that
$| \xi/\langle \xi \rangle - \eta / \langle \eta\rangle | > \delta$
 or $ | x- y |> \delta$. 
But then if the kernel of the operator in~\eqref{eq:TS1} is given by $K_M(x,\xi,y,\eta)$ where
$$
|(hD_{x,\xi})^\alpha K_N(x,\xi,y,\eta)|\leq C_{\alpha,N} e^{b\langle \xi\rangle -M\log \langle \eta\rangle -\psi(x,\xi,y,\eta)}.
$$
Now, fix $0<\delta<1$ small. Then, when $|\xi/\langle \xi\rangle -\eta/\langle \eta\rangle|>\delta$ or $|x-y|>\delta$, 
\begin{equation}
\label{e:estimateMeNow}
|\xi-\eta|^2 +\langle \xi\rangle \langle \eta\rangle |x-y|^2\geq \frac{\delta^2}{16} (\langle \xi\rangle +\langle \eta\rangle)^2.
\end{equation}
To see this, observe that on 
$$
\Big|\frac{\langle \xi\rangle -\langle \eta\rangle}{\langle \xi\rangle +\langle \eta\rangle}\Big|\geq \frac{\delta}{4},
$$
we have
$$
\frac{\delta}{4}\leq \Big|\frac{\langle \xi\rangle^2 -\langle \eta\rangle^2}{(\langle \xi\rangle +\langle \eta\rangle)^2}\Big|\leq \frac{|\xi-\eta|}{\langle \xi\rangle +\langle \eta\rangle}.
$$
On the other hand, when 
$$
\Big|\frac{\langle \xi\rangle -\langle \eta\rangle}{\langle \xi\rangle +\langle \eta\rangle}\Big|\leq \frac{\delta}{4},
$$
we have 
$$
\frac{2\langle \xi\rangle \langle \eta\rangle}{\langle \xi\rangle +\langle \eta\rangle}=\frac{\langle \xi\rangle+\langle \eta\rangle}{2}\Big(1-\Big[\frac{\langle \eta\rangle -\langle \xi\rangle}{\langle \xi\rangle +\langle \eta\rangle}\Big]^2\Big)\geq \frac{1}{4}(\langle \xi\rangle+\langle \eta\rangle)
$$
Therefore, if $|x-y|\geq \delta$,~\eqref{e:estimateMeNow} follows. If instead, $|\xi/\langle \xi\rangle -\eta/\langle \eta\rangle|\geq \delta$, then
$$
\frac{|\xi-\eta|}{\langle \xi\rangle +\langle \eta\rangle}\geq \frac{1}{2}\Big[\Big|\frac{\xi}{\langle \xi\rangle} -\frac{\eta}{\langle \eta\rangle}\Big|-\Big(\frac{|\xi|}{\langle \xi\rangle }+\frac{|\eta|}{\langle \eta\rangle}\Big)\Big|\frac{\langle \xi\rangle -\langle \eta\rangle}{\langle \xi\rangle +\langle \eta\rangle}\Big|\Big]\geq \frac{\delta}{4}
$$
and~\eqref{e:estimateMeNow} follows.

From~\eqref{e:estimateMeNow}, we have that there is $C_{M,\delta}>0$ such that if $|\xi/\langle \xi\rangle -\eta/\langle \eta\rangle|>\delta$ or $|x-y|>\delta$, 
\[ \begin{aligned}  b \langle \xi \rangle -& c( \langle \xi \rangle + \langle \eta\rangle )^{-1} \left(  | \xi - \eta|^2 
 +  \langle \xi \rangle \langle \eta \rangle | x - y |^2 \right)+M\log \langle \eta\rangle \\
   &\leq   b  \langle \xi \rangle  - \tfrac{1}{64}c\delta^2 (\langle \xi\rangle +\langle \eta\rangle)
    - \tfrac 12 c( \langle \xi \rangle + \langle \eta\rangle )^{-1} \left(  | \xi - \eta|^2 
 +  \langle \xi \rangle \langle \eta \rangle | x - y |^2 \right)+C_{M,\delta}, \end{aligned} \]
and the Schur criterion and gives 
\eqref{eq:TS1} for $b\leq \frac{c\delta^2}{64}.$} \end{proof}

\begin{proof}[Proof of Proposition \ref{p:awf}]
We start by recalling the characterization of the analytic wave front set using the standard FBI/Bargmann--Segal transform:
\[  \mathscr T u ( x, \xi ; h) := c_n h^{ -\frac {3n} 4 }
\int_{\RR^n } e^{ \frac i h ( \langle x - y , \xi \rangle  + 
\frac i 2 ( x - y)^2 ) }u(y)dy , \ \ \ u \in \mathscr S' ( \RR^n ) .\]
Then 
\begin{equation}
\label{eq:FBI_wf}
( x_0 , \xi_0 ) \notin \WFa ( u ) \ \Longleftrightarrow \ 
\left\{ \begin{array}{l} 
\exists \, \delta, \, U = {\rm{neigh}}((x_0, \xi_0)) \\
|\mathscr T u ( x, \xi , h ) | \leq C e^{ - \delta/h } , \  \  ( x, \xi ) 
\in U , \ \ 0 < h < h_0 . \end{array} \right.
\end{equation}
see \cite[Theorem 9.6.3]{H1} for a textbook presentation;  
note the somewhat different convention: $ \mathscr T u ( x, \xi;h ) 
= e^{ - \frac 1 {2h} \xi^2 } T_{1/h} u ( x - i \xi ) $.

We first prove \eqref{eq:awf}. Hence suppose that $ ( x_0, \xi_0 ) 
\in U \times \Gamma $. Let $ \chi \in S^{{0}} $ be supported in a small
conic neighbourhood, $ U_0 \times \Gamma_0 $, of $ ( x_0, \xi_0 ) $ and choose 
$ \chi_1 \in S^{{0}} $ which is supported in $ U \times \Gamma $
and is equal to $ 1 $ on a conic neighbourhood of the support of $ 
\chi $ {and $\chi_2\in S^0$ supported in $U \times \Gamma$ and equal to 1 on a conic neighborhood of the support of $\chi_1$}. Our assumptions then show that $ e^{ a \langle \xi \rangle /h } {\chi_2} Tu \in L^2 ( \RR^{2n} ) $ for some $ a > 0 $. 
We now write
\[  \chi e^{ b \langle \xi \rangle } T u = 
\chi e^{ b \langle \xi \rangle } T S \left(  \chi_1 e^{ - a \langle \xi 
\rangle  } e^{  a \langle \xi \rangle } {\chi_2}Tu + 
( 1 - \chi_1 ){\langle \xi\rangle^{N}\langle \xi\rangle^{-N}} T u \right) . \]
Since $u\in H^{-N}$, $ {\langle \xi\rangle^{-N}}T u \in L^2 ( \RR^{2n } ) $ and
  \eqref{eq:TS}, \eqref{eq:TS1},  now
show that $ e^{ b \langle \xi \rangle} \chi T u \in H^K $ 
for some $ b > 0 $ and any $ K $. By taking $ K > n $ and applying
\cite[Corollary 7.9.4]{H1} we obtain 
a uniform bound 
\[    |T u ( x, \xi )| \leq  C e^{ - b \langle \xi \rangle } , \ \
( x, \xi ) \in U_0 \times \Gamma_0. \]
Let $ h_1 $ be the fixed $ h $ in the definition of $ T $. 
Then, 
\begin{equation} 
\label{eq:T2T}
\mathscr T ( x, \xi/\langle \xi \rangle; 
h_1/\langle \xi \rangle )  = T u ( x, \xi ) = 
\mathcal O ( e^{ - b \langle \xi \rangle } ), \ \ ( x, \xi ) \in 
U_0 \times \Gamma_0 
. \end{equation}
Putting $ \omega_0 := \xi_0/\langle \xi_0 \rangle $, it follows that
$ \mathscr T ( x, \omega , h ) = \mathcal O ( e^{ - \delta/h } ) $ 
for $ ( x, \omega ) $ in a small neighbourhood of $ ( x_0 , \omega_0 ) $.
But then \eqref{eq:FBI_wf} shows that $ ( x_0, \omega_0 )  \notin 
\WFa ( u ) $. Since $ \WFa ( u ) $ is a closed conic set, 
we conclude that $ ( x_0, \xi_0 ) \notin \WFa ( u ) $.

Now suppose that $ \WFa ( u ) \cap ( U \times \Gamma ) = \emptyset $.
Then for $ ( x, \omega ) $ near $ U_0 \times (\Gamma_0 \cap \mathbb S^{n-1}) $ (with $ U_0 $ and $ \Gamma_0 $, 
as in the statement of the theorem), 
$ \mathscr T ( x, \omega, h ) = \mathcal O ( e^{ - \delta / h } ) $. 
Reversing the argument in \eqref{eq:T2T} we see that
\[  | T u ( x, \xi ) | \leq C e^{ - b \langle \xi \rangle } , \ \ 
( x, \xi ) \in U_0 \times \Gamma_0 . \]
{Now, since $u\in H^{-N}(\mathbb{R}^n)$, $\langle \xi\rangle^{-N}Tu\in L^2(\mathbb{R}^{2n})$. }In particular, since $ | \psi | \leq C \langle \xi \rangle $ and the support of 
$ \psi $ is contained in $ U_0 \times \Gamma_0 $, \eqref{eq:cawf}
follows.
\end{proof}

The next proposition relates weighted estimates to deformed FBI
transform: 
\begin{prop}
\label{p:weig2def}
Suppose that $ H_\Lambda $, $ \Lambda = \Lambda_G $,
is defined in \cite[(4.7)]{gaz1} with $ G $ satisfying
\eqref{e:LambdaDef}
with $ \epsilon_0 $ chosen as in the definition of $ H_\Lambda $.

Then there exists $ \psi \in S^1 ( T^* \RR^n ) $
such that $ T : \mathscr B_\delta \to L^2 ( T^* \RR^n , e^{  \delta \langle
\xi \rangle /Ch } d x d \xi  ) $ extends to 
\begin{equation}
\label{eq:H2L2}   T =  \mathcal O ( 1 ) : H_\Lambda \to L^2 ( T^* \RR^n , e^{2 \psi ( x,\xi)/h } d x d \xi ) , \end{equation}
and $ S : L^2 ( T^* \RR^n , e^{ {-C} \delta \langle
\xi \rangle /h } d x d \xi  )  \to \mathscr B_\delta $, extends to 
\begin{equation}
\label{eq:L22H}
S = \mathcal O ( 1 ) : L^2 ( T^* \RR^n , e^{2 \psi ( x,\xi)/h } d x d \xi )
\to H_\Lambda .
\end{equation}
In addition, 
\begin{equation}
\label{eq:weig2def}  \psi ( x, \xi ) = G ( x, \xi ) + 
\mathcal O ( \epsilon_0^2 )_{ S^1 ( T^* \RR^n ) } .
\end{equation}
\end{prop}

For a simpler version of this result in the case of compactly supported weights see \cite[\S 8]{gaz0}.

\begin{proof}
The statement \eqref{eq:H2L2} is equivalent to 
\[  T S_\Lambda = \mathcal O ( 1 ) :  L^2 ( \Lambda , e^{ - 2 H (\alpha)/h } d \alpha ) \to L^2 ( T^* \RR^n , e^{ 2 \psi ( \beta ) } d\beta )
\]
and hence we analyse the kernel of the operator  $ T S_\Lambda $ which
is given by  
\begin{gather*} K ( \alpha, \beta ) = 
c_n h^{ - \frac {3n} 2} 
\int_{\RR^n} 
e^{ \frac i h ( \varphi_0 ( \alpha, y ) + \varphi_G^* ( \beta, y ) )}
\langle \beta_\xi  \rangle^{\frac n 4} \langle \alpha_x \rangle^{\frac n 4}
( 1 + \tfrac i 2 \langle \alpha_x - y \rangle ) dy , 
\end{gather*}
where the notation (and also notation for $ \Phi $ below) comes from 
\eqref{eq:lost}.
The integral in $ y $ converges and can be evaluated by a completion of
squares as in \cite[Proposition 4.4]{gaz1}. That gives the phase
\eqref{eq:Psi} with
$ \alpha \in T^* \RR^n $ and $ \beta \in \Lambda $.
The critical point in $ y $ is given by 
\begin{equation}
\label{eq:ycr}  y_c ( \alpha, \beta ) = \frac 1 {\langle \alpha_\xi\rangle +\langle \beta_\xi\rangle} \left( \langle \alpha_\xi \rangle \alpha_x + 
\langle \beta_\xi \rangle \beta_x + i ( \beta_\xi - \alpha_\xi ) \right) .\end{equation}

We then have \eqref{eq:H2L2} with 
\begin{equation}
\label{eq:psia}  \psi ( \alpha ) :=  \max_{ \beta \in \Lambda } \left( - \Im 
\Psi ( \alpha, \beta ) + H ( \beta ) \right) . \end{equation}
We have (see \cite[(3.3),(3.4)]{gaz1})
\[  d_\beta ( - \Im  \Psi ( \alpha, \beta ) + H ( \beta ) ) = 
 \Im ( -  \partial_{z,\zeta}  \Psi ( \alpha, (z,\zeta)  )  - 
 \zeta d z |_{ \Lambda } )|_{ (z , \zeta ) = \beta \in \Lambda } . \]
Now, if $ y_c ( \alpha, (z,\zeta) ) $ is the critical point in $ y $, then
\[ \begin{split} \partial_{z,\zeta}  \Psi ( \alpha, z  ) & = \partial_{ z, \zeta }
( \Phi ( \alpha, y_c ( \alpha, (z,\zeta) ) ) - \bar \Phi ( (z,\zeta), y_c ( \alpha , (z,\zeta) ) ))
= -\partial_{z , \zeta } \bar \Phi\big|_{y=y_c(z,\zeta)}(z,\zeta) \\
& = 
- \zeta \cdot d z + ( y_c - z ) \cdot d \zeta + i
\langle \zeta \rangle ( z - y_c ) \cdot dz + \tfrac i 2 ( z - y_c)^2 \zeta
\cdot d \zeta / \langle \zeta \rangle . 
\end{split} 
\]
For $ G = 0$ the critical point (see \eqref{eq:ycr}) is given by 
$ \alpha = \beta $. Hence 
\begin{equation}
\label{eq:betac} \beta_c = \beta_c ( \alpha ) = \left( \alpha_x + \mathcal O ( \epsilon_0 )_{S^0} , \alpha_\xi + \mathcal O ( \epsilon_0 )_{S^1 }
\right) , \end{equation}
with $ \epsilon_0 $ as in \eqref{e:LambdaDef}.

Hence we obtain $ \psi $ by inserting the critical point 
$ \beta_c  $ into the right hand side of
\eqref{eq:psia}
\begin{equation}
\label{eq:psiak}  \psi ( \alpha ) = - \Im \Psi ( \alpha, \beta_c ( \alpha ) ) + H ( \beta_c ( \alpha ) ) \in S^1 ( T^* \RR^n )  . 
\end{equation}
(We note that for $ G = 0 $ the maximum in \eqref{eq:psia} is non-degenerate and unique and it remains such under small symbolic perturbations.) From \eqref{eq:Psi} we see that
\[  \Im \Psi ( \alpha, \beta_c ( \alpha ) ) = \Im \Psi ( \alpha, 
\alpha + \mathcal O ( \epsilon_0 )_{ S^0 \times S^1 } ) = 
\alpha_\xi \cdot G_\xi ( \alpha ) + \mathcal O ( \epsilon_0^2 )_{ S^1} .\]
Inserting this into \eqref{eq:psiak} and recalling that $ 
H = \xi G_\xi -G$ we obtain \eqref{eq:weig2def}.

To obtain \eqref{eq:L22H} we apply the same analysis to $ T_\Lambda S $
and we need to show that two weights coincide. That is done as in 
\cite[\S 8]{gaz0}.
\end{proof}



\section{Proof of Theorem \ref{t:2}}

As already indicated in \S \ref{s:idea}, to prove the theorem we construct a family of weights $ G_\epsilon \in S^1 $,
uniformly bounded in $ S^1 $, supported in a conic neighbourhood of $ \Gamma = \{ ( 0 , 0 , \xi_1 , 0 ) : 
\xi_1 > M \} $, $ M \gg 1 $,  and  
satisfying $ 0 \leq G_\epsilon \leq C_\epsilon \log \langle \xi \rangle $.
In addition, 
\begin{equation}
\label{eq:propG0}  H_p G_\epsilon \geq 0  , \  \ \  
\text{  $ G_\epsilon \to \xi_1 $ on $ \Gamma $ (in $ S^{1+} $)}, 
\end{equation} 
with $ H_p G_\e \gg \xi_1^{m-1} $ in a suitable sense (see \eqref{eq:propHpG3})
for $ \e \ll 1 $.

We will then put $ \Lambda_\epsilon := \Lambda_{G_\epsilon } $ so that
the assumption $ u \in \CI $ will give 
$ u \in H_{\Lambda_\epsilon } $. On the other hand the assumption that
$ \Gamma \cap \WFa ( Pu ) $ shows that $ 
\| Pu \|_{ H_{\Lambda_\epsilon } } \leq C $ with the constant $ C $ independent of $ \e $. But then \cite[Proposition 6.2]{gaz1} and the 
properties of $ G_\epsilon $ show that $ \| u \|_{H_{\Lambda_\epsilon}} $ is bounded independently of $ \epsilon $. Propositions \ref{p:awf}
and \ref{p:weig2def} then show that $ \WFa ( u ) \cap \Gamma_0 = \emptyset $. 

\subsection{Construction of the weight}
We now construct a family of weights, $ G_\epsilon $, satisfying 
\eqref{eq:propG0}. In fact, we need more precise conditions on 
$ G_\epsilon $ given in the following 

\begin{lemm}
\label{l:G}
Suppose that $ p $ satisfies \eqref{eq:hypP} at $ \rho_0 = (x_0, \xi_0 )
\in T^* \RR^n \setminus 0  $ {and $\Gamma$ is an open conic neighbourhood of $\rho_0$}. 
Then, there exists $ G_\epsilon \in S^1 ( T^*\RR^n ) $, $ \supp G_\epsilon \subset 
 \Gamma $, such that 
\begin{equation}
\label{eq:propHpG1} 
\begin{gathered}  
| \partial_x^\alpha \partial_\xi^\beta  G_\epsilon | \leq C_{\alpha \beta}
\langle \xi \rangle^{ 1 - |\beta| } , \ \  
0 \leq G_\epsilon \leq C \epsilon^{-1} \log \langle \xi \rangle, 
\\   G_\epsilon ( x, \xi ) |_{ 1 \leq |\xi| \leq 1/\epsilon } = \Phi ( x, \xi ) |\xi| , 
 \ \  \Phi \in S^0_{\rm{phg}} ( T^* \RR^n ) , \ \ 
 \Phi ( x_0, t \xi_0 ) = 1 , \  t \gg 1, 
\end{gathered}
\end{equation} 
\begin{equation}
\label{eq:propHpG2} 
{ H_p G_\e ( x , \xi) \geq c_0  
\left( 
 \langle \xi \rangle^m |\partial_\xi G_\e ( x, \xi ) |^2 +
 \langle \xi \rangle^{m-2} |  \partial_x G_\e (x, \xi)|^2 \right) , }
\end{equation}
\begin{equation}
\label{eq:propHpG3} { \forall \, M_1, \, \gamma \geq 0 \, \, \exists \, M_2, \, K,  \, \epsilon_0
\, \forall \, 0 < \epsilon < \epsilon_0 , \ \ \   H_p G_\epsilon e^{ \gamma G_\epsilon} + M_2\langle \xi \rangle^{K } \geq M_1 
\langle \xi \rangle^{m-1} e^{ \gamma G_\epsilon }} .
\end{equation}
\end{lemm}
We stress that the constants $ C_{\alpha \beta} $ and $ c_0 $ are independent of 
$ \epsilon $ and $ M_1 $.
\begin{proof} We use the normal form for $ p $ constructed in 
\cite[\S 3]{hab}. That means that we take $ x_0 = 0 $ and $ \xi_0 = e_1 := (1, 0 , \cdots, 0 ) $ and 
 can assume that $ p ( x, \xi ) = -  
\xi_1^{ m} x_1 $ in a conic neighbourhood of $ \rho = ( 0 , e_1 ) $.
For simplicity we can assume that $ m = 1 $ as the argument is the same otherwise.

Let $\chi \in C_c^\infty( \RR ;[0,1])$ satisfy
\begin{equation}
\label{eq:defchi}  \supp \chi \subset [ - 2,2 ] , \ \ 
 \chi_{ |t | \leq 1 } = 1, \ \  
 t\chi'( t )\leq 0. \end{equation}
and put $ \varphi ( t ) := \chi ( t/\delta ) $. 
Here $ \delta $ will be fixed depending on $ \Gamma $. Using this function we define
$ \Phi = \Phi ( x, \xi ) := 
\varphi_1 \varphi_2 \varphi_3 \psi  $ where 
\begin{equation}  
\label{eq:defph}
\varphi_1 := \varphi ( x_1 ) , \ \ 
\varphi_2 := \varphi( |\xi'|/\xi_1  ) 
\ \ \varphi_3 = \varphi ( |x'| ) , \ \ 
\psi :=  ( 1 - \varphi ( (\xi_1)_+  ) ) . 
\end{equation}
We choose $ \delta $ small enough so that $ \supp \Phi \subset \Gamma $.

We define $ G_\e $ as follows
\begin{equation}
\label{eq:defG}
G_{\e} ( x, \xi ) =\Phi ( x, \xi )  q_\e ( \xi_1 )   , \ \ \ 
q_\e (t) := \int_0^t \left( \chi( \epsilon s )  + ( 1 - \chi ( \epsilon s ))
( s \epsilon )^{-1} \right)  ds .
\end{equation} 
We check that 
\begin{equation}
\label{eq:q123}
\begin{gathered}
\xi_1 \partial_{\xi_1} q_\e \geq \min ( \xi_1, \e^{-1} ), \\
 \xi_1 \indic_{\xi_1 \leq 1/\epsilon } + 
\epsilon^{-1}   ( 1 + \log ( \epsilon \xi_1  ) ) \indic_{\xi \geq 1/\epsilon }  \leq q_\epsilon\leq \xi_1 \indic_{\xi_1 \leq 1/\epsilon } + 
\epsilon^{-1}  ( 2 + \log ( \epsilon \xi_1  ) )\indic_{\xi \geq 1/\epsilon }   .
\end{gathered}
\end{equation}

Uniform boundedness of $G_{\e}$ in $S^1$ means that $ q_\e$ in \eqref{eq:defG} satisfies $ 
| \partial^k_{\xi_1 } q_\e | \leq C_k \xi_1^{1-k} $ with $ C_k$'s independent of $ \e$. But this is immediate from the definition. 
We also easily see that $ G_\e $ converges to $G :=  \Phi ( x, \xi ) \xi_1 $ 
 in $S^{1+}$ as $ \e \to 0 $. This proves \eqref{eq:propHpG1}.

To see \eqref{eq:propHpG2}, we first note that, since $ \Phi \geq 0$, 
$ \Phi \in S^0 $, 
the standard estimate $ f ( z) \geq 0 \Longrightarrow |d  f( z )|^2 \leq
C f ( z ) $ gives,
\begin{equation}
\label{eq:Phineq} 
\Phi ( x, \xi ) \geq c_1 \left(
\xi_1^2 | \partial_\xi \Phi ( x , \xi)|^2 + | \partial_x 
\Phi ( x , \xi )|^2 \right) . 
\end{equation}
Note also that we have $ H_p =    \xi_1 \partial_{ \xi_1 } -                                   
x_1 \partial_{x_1} $ and therefore
\begin{equation}
\label{e:HpPhi}  H_p \Phi =  - x_1 \varphi' ( x_1) \varphi_2 \varphi_3 \psi -
 ( |\xi'|/\xi_1 ) \varphi' ( |\xi'|/\xi_1 ) \varphi_1 \varphi_3 \psi
- \varphi_1 \varphi_2 \varphi_3 \xi_1 \varphi' ( (\xi_1)_+ )  \geq 0. \end{equation}
{Since $ q_\epsilon \in S^1 $,  $ \xi_1 \partial_{\xi_1 } q_\e ( \xi_1 )  \geq c_2 
 \xi_1 ( \partial_{\xi_1 } q_\e( \xi_1 ) )^2  $. We also claim that 
\begin{equation}
\label{eq:xi1q}  \xi_1 \partial_{\xi_1 } q_\e ( \xi_1 )  
\geq 
 c_2 \xi_1^{-1} q_\e ( \xi_1 )^2 . \end{equation}
In fact, using \eqref{eq:q123} we see  that  {to prove} \eqref{eq:xi1q}  {it is enough to have} 
\[  \min ( t, \epsilon^{-1} ) \geq  {c_2} 
t^{-1} \left( t \indic_{ t \leq 1/\epsilon } ( t )  + 
\epsilon^{-1} (  {2} + \log ( t \epsilon ) ) \indic_{ t \geq 1/\epsilon  }
(t) \right)^2 . \]
This clearly holds (with $c_2=1$) for $ t \leq 1/\epsilon $ and for $ t \geq \epsilon $
is equivalent to
$ {c_2} (  {2} + \log s )^2 \leq  s  $, $ s = t \epsilon \geq 1 $, 
which holds with $  {c_2} = { \frac 1 4} $.
It follows that
\[  \xi_1 \partial_{\xi_1 } q_\e ( \xi_1 )  \geq 
c_2 \left( \xi_1^{-1} q_\e ( \xi_1 ) ^2 + \xi_1 ( \partial_{\xi_1 } q_\e( \xi_1 ) )^2 \right), \]
 which combined with \eqref{eq:Phineq} and~\eqref{e:HpPhi} gives
\[ \begin{split} H_p G_\e & = \Phi ( \xi_1 \partial_{\xi_1 } q_\e ) + 
(H_p \Phi) q_\e \\
& \geq \Phi ( \xi_1 \partial_{\xi_1 } q_\e ) 
\geq c_2  \xi_1 \Phi ( \partial_{\xi_1 } q_\e )^2 + 
c_3 \left( \xi_1^2 | \partial_\xi \Phi |^2 + | \partial_x 
\Phi|^2 \right) \xi_1^{-1} q_\e^2 \\
& \geq c_0 \left( \xi_1 |\partial_\xi G_\e|^2 + 
\xi_1^{-1} | \partial_x G_\e|^2 \right).
\end{split}
\]
Since $ \langle \xi \rangle \sim \xi_1 $ on the support of $ G_\e $, 
we obtain  \eqref{eq:propHpG2}.}

Finally we prove \eqref{eq:propHpG3}. Since by~\eqref{e:HpPhi} we have $H_pG_\e \geq \Phi H_pq_\e$, we see that \eqref{eq:propHpG3} follows from proving that for any $ M_1 $ we can find 
$ K $, $ M_2 $ and $ \epsilon_0  $ such that for $ \xi_1 \geq 1 $, 
\begin{equation}
\label{eq:PhiHp} \Phi H_p q_\epsilon e^{ \gamma \Phi q_\epsilon } + M_2  \xi_1^{K }  
\geq M_1 e^{ \gamma \Phi q_\epsilon } . \end{equation}
Using \eqref{eq:q123}, we see that for 
$ \xi_1 \leq 1/\epsilon $ we need
$
G_\e e^{\gamma G_\e}+M_2\xi_1^K\geq M_1e^{\gamma G_\e} 
$.
This holds for 
\[ K = 0, \ \  M_2 = 2\gamma^{-1} e^{\gamma M_1-1} \]
since for $\gamma>0$ and $a \geq 0$,
$ a e^{\gamma a} - M_1 e^{ \gamma a } \geq -2 \gamma^{-1} e^{\gamma M_1 -1} $.

For $ \xi_1 \geq 1/\e $, we need to find $ K $ and $ M_2 $ for which
\begin{equation}
\label{eq:M1M2N}  \epsilon^{-1} \Phi e^{ \gamma \Phi q_\epsilon }
+ M_2 \xi_1^K \geq M_1 e^{ \gamma \Phi q_\epsilon } .
\end{equation}
Using $ a e^{ a b } +  M_1 e^{ M_1 b }  \geq M_1 e^{a b } $
with $ a := \epsilon^{-1} \Phi $ and 
$$ b :=   {\gamma \epsilon q_\epsilon} 
\leq \gamma (   {2} + \log (  {\epsilon }\xi_1) )\leq  { \gamma (2+\log \xi_1)}  ,$$ 
we obtain 
\eqref{eq:M1M2N} with $ M_2 = M_1 e^{   {2\gamma} M_1 } $ and 
$ K = \gamma M_1 $. Hence we obtain \eqref{eq:PhiHp} proving  
\eqref{eq:propHpG3}.
\end{proof} 

\subsection{Microlocal analytic hypoelliticity}
We will have bounds which are uniform in $ \epsilon $ but not in $ h$. 
{We start with the following
\begin{lemm}
\label{l:new}
Suppose that $ P $ is of the form \eqref{eq:defP} with real valued
principal symbol $ p $ and suppose that $ \Gamma \subset 
U \times \RR^n \setminus $ is an open cone, $ \Gamma \cap 
\mathbb S^{n-1} \Subset U \times \mathbb S^{n-1}$
and 
\begin{equation}\begin{gathered}
\label{eq:newG}  G \in S^1 ( \Gamma ; \RR) , \ \ 
| G | \leq C \log \langle \xi \rangle 
, \\
H_p G ( x , \xi) \geq c_0  \left(
 \langle \xi \rangle^m |\partial_\xi G ( x, \xi ) |^2 +
 \langle \xi \rangle^{m-2} |  \partial_x G (x, \xi)|^2\right) .\end{gathered}
\end{equation}
 Then for $ T_\Lambda $, $ H_\Lambda $, $ \Lambda = \Lambda_{\theta G } $ defined in
\eqref{e:LambdaDef} and \eqref{eq:HLa}, $ h $ and $ \theta $ sufficiently small, and $ u \in
H_\Lambda^{-N+ {m}} $,
\begin{equation}
\label{eq:new}
\begin{gathered}
\Im \langle h^m P u ,  u \rangle_{ H^{{-N}}_\Lambda } \ \geq \tfrac12 \theta
\langle  H_p G \, \langle \xi\rangle^{-N}T_\Lambda u ,\langle \xi\rangle^{-N} T_\Lambda u \rangle_{ 
L^2_\Lambda }   - M h \| 
 u \|^2_{ H^{\frac{m-1}2 -N  }_\Lambda } , 
\end{gathered}
\end{equation}
where $ M $ depends only on $ P $ and the semi-norms of $ G $ in $ S^1 $.
\end{lemm}}
\begin{proof}
We use Proposition \ref{p:6.2} and~\cite[Proposition 6.3]{gaz1} to see that for any $K>0$,
\begin{equation}
\label{eq:l3} \begin{split}  \Im \langle h^m P u , u \rangle_{H^{-N}_\Lambda } 
& = {\Im}\langle {\langle \xi\rangle^{-2N}}T_\Lambda h^m P S_\Lambda T_\Lambda u , T_\Lambda u \rangle_{L^2_\Lambda } \\
&= {\Im}\langle{\Pi_\Lambda \langle \xi\rangle^{-2N}} \Pi_\Lambda 
h^m P S_\Lambda \Pi_\Lambda T_\Lambda u , T_\Lambda u \rangle_{L^2_\Lambda }
\\
& = \langle (\Im b_{P,N} ) T_\Lambda u , T_\Lambda u\rangle_{L^2_\Lambda} + \mathcal O ( h^\infty ) 
\| u \|_{ H_\Lambda^{-K} }\\
&  \geq  
\langle (\Im p|_\Lambda )\, {\langle \xi\rangle^{-N}} T_\Lambda u , {\langle \xi\rangle^{-N}} T_\Lambda u \rangle_{L^2_\Lambda } 
- M h  \| u \|_{ H_\Lambda^{\frac{m-1}2 - N} } . 
\end{split} \end{equation}
From \eqref{eq:p6.2} and \eqref{eq:newG} we obtain
\[ \begin{split} \Im p|_\Lambda & = \Im p ( x - i \theta \partial_\xi G ( x, \xi ) , 
\xi + i \theta \partial_x G ( x, \xi )) \\
& = 
\theta H_p G ( x, \xi ) + \theta^2 \mathcal O \left(  \langle \xi \rangle^m   
|\partial_\xi G ( x, \xi ) |^2 + \langle \xi \rangle^{m-2} |\partial_x
G ( x, \xi ) |^2 \right) \\
& \geq \tfrac12  \theta H_p G ( x, \xi ) , 
\end{split} \]
if $ \theta $ is small enough. 
\end{proof}

The next lemma allows us to use smoothness of $ u $ to obtain weaker
weighted estimates:
\begin{lemm}
\label{l:nowe}
Suppose $ U \subset \RR^n $ is an open set, 
\[ G \in S^1 ( T^* \RR^n ) ,  \ \ G \geq 0 , \ \ 
\supp G \subset K \times \RR^n , \ \ K \Subset U , 
\]
and  $ T_\Lambda $, $ H_\Lambda $, $ \Lambda = \Lambda_{\theta G } $ are defined in
\eqref{e:LambdaDef} and \eqref{eq:HLa}. Then, there exists $a>0$ such that for every $\chi,\tilde{\chi}\in S^1$ with $\tilde{\chi}\equiv 1$ in a conic neighborhood of $\supp \chi$ and every $K,N>0$, there exists $c, C>0$ such that for all $u\in H^{-N}(\mathbb{R}^n)$,
\begin{equation}
\label{eq:nowe}  \| \langle \xi \rangle ^{K} e^{ - a G/h } \chi T_\Lambda u \|_{ L^2_\Lambda } \leq 
C( \| \langle \xi\rangle^{K}\tilde{\chi} T u \|_{L^2(T^*\mathbb{R}^n)}+ e^{-c/h}\|\langle \xi\rangle^{-N}Tu\|_{L^2(T^*\mathbb{R}^n)})  . \end{equation}
In particular, if $\chi\equiv 1$ on $\supp G$, then
\begin{equation}
\label{eq:nowe2}
\begin{split}
& \|(\langle \xi\rangle^Ke^{-a/h}\chi +\langle \xi\rangle^{-N}(1-\chi))T_\Lambda u\|_{L^2_\Lambda} \\
& \ \ \ \ \ \ \leq C( \| \langle \xi\rangle^{N}\tilde{\chi} T u \|_{L^2(T^*\mathbb{R}^n)}+ e^{-C/h}\|\langle \xi\rangle^{-N}Tu\|_{L^2(T^*\mathbb{R}^n)}).
\end{split}
\end{equation}
\end{lemm}
\begin{proof}

First, observe that by~\cite[Lemma 4.5]{gaz1}, for any $\delta>0$,
\[ T_{\Lambda}S = K_\delta + O_{N,\delta}(e^{-c_\delta /h})_{\langle \xi\rangle^{N}L^2(T^*\mathbb{R}^n)\to \langle \xi\rangle^{-N}L^2_\Lambda}, 
\]
and $K_\delta $ has kernel, $ K_\delta (\alpha,\beta)$, given by 
$$
\begin{gathered}
h^{-n}e^{\frac{i}{h}\Psi(\alpha,\beta)}k(\alpha,\beta)\psi(\delta^{-1}|\Re \alpha_x- \beta_x|) )\psi(\delta^{-1}\min(\langle\Re \alpha_\xi \rangle,\langle \beta_\xi\rangle)^{-1}|\Re \alpha_\xi-\beta_\xi|), 
\end{gathered}
$$
where $ (\alpha, \beta ) \in \Lambda \times T^*\mathbb{R}^n  $ and $\Psi$ is as in~\eqref{eq:Psi}, and $\psi\in C_c^\infty(\mathbb{R})$ is identically 1 near 0. Therefore, we need only consider $K_\delta(\alpha,\beta)$.

To do this, let $\tilde{\chi}\in S^0$ be identically 1 on a conic neighborhood of $\supp\chi$. Then, for $\delta>0$ small enough,
$$
\chi(\Re \alpha)K_\delta(\alpha,\beta)(1-\tilde{\chi})(\beta)\equiv 0.
$$
Therefore, 
$$
\chi e^{-a G/h}\langle \xi\rangle^K T_\Lambda S(1-\tilde{\chi})=O_{N}(e^{-c/h})_{\langle \xi\rangle^{N}L^2(T^*\mathbb{R}^n)\to \langle \xi\rangle^{-N}L^2_\Lambda}.
$$

For the mapping properties 
\[ \chi e^{-aG/h}T_\Lambda S\tilde{\chi} : \langle \xi\rangle^{-K} L^2(T^*\mathbb{R}^n)\to \langle \xi\rangle^{-K}L^2_\Lambda, \]
we consider the operator
$$
\chi e^{-aG/h}e^{-H/h}\langle \xi\rangle^K T_\Lambda S\tilde{\chi}\langle \xi\rangle^{-K}:L^2(T^*\mathbb{R}^n)\to L^2(\Lambda; dxd\xi).
$$
Modulo negligible terms, the kernel of this operator is given by 
$$
h^{-n}e^{\frac{i}{h}(\varphi((x,\xi),(y,\eta)))}\tilde{k}((x,\xi),(y,\eta))
$$
where $\tilde{k}\in S^0$ has
\begin{equation}
\label{e:suppPropK}
\supp \tilde{k} \subset\{ |\xi-\eta|\leq C\delta \langle \xi\rangle \}\cap \{|x-y|\leq C\delta\}.
\end{equation}
and 
$$
\varphi= iH(x,\xi)+ia\theta G(x,\xi)+\Psi((x-i\theta G_\xi,\xi+i\theta G_x(x,\xi)),(y,\eta)),
$$
with $H(x,\xi)=\theta \langle \xi,G_\xi(x,\xi)\rangle -\theta G(x,\xi).$ Using~\eqref{e:suppPropK}, we have 
\begin{align*}
\Im\varphi &= aG +\theta \xi\cdot G_\xi-\theta G+\frac{\langle \eta\rangle\langle\xi\rangle}{2(\langle \eta\rangle+\langle \xi\rangle)}\left( (x-y)^2-(\theta G_\xi)^2 \right) + \frac{ (\xi-\eta)^2-(\theta G_{\xi})^2}{2(\langle \eta\rangle+\langle \xi\rangle)}\\
&\qquad  + \theta \xi\cdot G_\xi +O(\theta(|x-y| |G_x|+\langle \xi\rangle^{-1}|\xi-\eta||G_\xi|))\\&\qquad \qquad+O(\theta^2(\langle \xi\rangle^{-1}|G_x|^2+\langle \xi\rangle |G_\xi|^2))\\
&\geq (a-\theta)G-C\theta^2(\langle \xi\rangle^{-1}(G_x)^2+\langle \xi\rangle|G_\xi|^2) +c\langle \xi\rangle (x-y)^2+c\langle \xi\rangle^{-1}(\xi-\eta)^2.
\end{align*}
In particular, taking $a$ large enough and using that $G\geq 0$, $G\in S^1$, (see the argument for \eqref{eq:Phineq}), we have
$$
\Im \varphi\geq \frac{a}{2}G(x,\xi)+c\langle \xi\rangle (x-y)^2+c\langle \xi\rangle^{-1}(\xi-\eta)^2.
$$
Therefore,  applying the Schur test for $L^2$ boundedness completes the proof that 
$$
\chi \langle \xi\rangle ^Ke^{-aG/h}T_\Lambda S\langle \xi\rangle^{-K}=O(1):L^2(T^*\mathbb{R}^n)\to L^2_\Lambda
$$
and the lemma follows.
\end{proof}

With these two lemmas in place we can prove the main result:
\begin{proof}[Proof of Theorem \ref{t:2}]
By multiplying $ u $ by a $ \CIc $-function  which is $ 1 $ in a neighbourhood of $ x_0 $, we can assume that $ u \in H^{-N+m}$, 
for some $ N $, is  compactly supported in $U$ and $\rho_0:=(x_0,\xi_0)\notin \WF(u).$ By Proposition \ref{p:FBI2H}, 
there exists $\tilde{\chi}\in S^0$ with $\tilde{\chi} \equiv 1$ in an open conic neighborhood, $\Gamma$, of $\rho_0$ such that for any $K>0$,
\begin{equation}
\label{e:smoothPart}
\|\langle \xi\rangle^K\tilde{ \chi} Tu\|_{L^2}\leq C_K.
\end{equation}
Also,  since $u\in H^{-N+m}$, 
\begin{equation}
\label{e:roughPart}
\|\langle \xi\rangle^{-N+m} Tu\|_{L^2}\leq C.
\end{equation}
Let $\Gamma_1\Subset \Gamma$ be an open conic neighborhood of $\rho_0$ and $\chi\in S^1$ with $\chi\equiv 1$ on $\Gamma_1$ and $\supp \chi \subset \Gamma$.

We choose $ \theta $ small enough so that
\eqref{e:LambdaDef} and \eqref{eq:l3} hold. We then fix $0< h\leq 1 $ small enough
so that \eqref{eq:l3} holds. From now we neglect the dependence on $ h $
which is considered to be a fixed parameter. We choose for
$ G = G_\e $ constructed in Lemma \ref{l:G} and supported in $\Gamma_1$. We recall that the estimates
depend only on the $ S^1 $ seminorms of $ G$ and these are uniform in 
$ \epsilon $.
We now claim that
$$
u\in H_{\Lambda_\e}^{-N+m},\qquad \Lambda_\e:=\Lambda_{\theta G_\e}.
$$
In fact, we can use \eqref{eq:nowe2} together with~\eqref{e:smoothPart} and~\eqref{e:roughPart}, observing that 
$ \exp ( a G_\e / h ) = \mathcal O_\epsilon ( \langle \xi \rangle^{
C a/ ( h \e ) }) $ and taking $K=C a/(h\e)$.

Next, note that $Pu\in H^{-N}$ is supported in $U$ and $\rho_0\notin \WF_a(Pu)$ .Propositions \ref{p:awf} and \ref{p:weig2def}
 (see \eqref{eq:cawf} and \eqref{eq:L22H} respectively) 
then show that for $ G_\e $ satisfying the assumptions of Lemma \ref{l:new} and $ \theta $ sufficiently small
$\| P u \|_{H^{-N}_{\Lambda_\epsilon } } \leq C_0 $,  
where $ C_0 $ depends only on $ P u $ and $ S^1$-seminorms of 
$ \theta G_\e $.

We now apply \eqref{eq:new} to obtain with $\Lambda_\e$ as above,
\begin{equation}
\label{eq:delta}
\tfrac12 \| u \|^2_{H^{-N}_{\Lambda_\e}}  + 
2  C_0^2 \geq \langle ( \theta H_p G_\epsilon - M \langle
\xi \rangle^{m-1} ) \langle \xi\rangle^{-N-m}  T_{\Lambda_\e} u,  \langle \xi\rangle^{-N}T_{\Lambda_\e} u \rangle_{ L^2_{\Lambda_\e}}  , 
\end{equation}
Let $ a $ be given by Lemma \ref{l:nowe} (so that \eqref{eq:nowe} 
holds). Then by \eqref{eq:propHpG3}, there exist $ M_2 $ and $ K$ such that
\[  \theta H_p G_\epsilon + M_2 \langle \xi \rangle^{2K} e^{ -2 a G_\epsilon /h } \geq ( M+ 1 ) \langle \xi \rangle^{m-1} . \]

From~\eqref{eq:nowe} we have
\begin{equation}
\label{e:intermediate}
\begin{aligned}
&\|M_2\chi \langle \xi\rangle^{K}e^{-aG_\e/h}\langle \xi\rangle^{-N}T_\Lambda u\|^2_{L^2_{\Lambda_\e}}\\
&\leq C(\|\langle \xi\rangle^{K-N}\tilde{\chi}Tu\|^2_{L^2(T^*\mathbb{R}^n)}+\|\langle \xi\rangle^{-N}Tu\|^2_{L^2(T^*\mathbb{R}^n)})\leq C^2_1
\end{aligned}
\end{equation}
Therefore, adding~\eqref{e:intermediate} to~\eqref{eq:delta}, and using that $\supp G_\e\subset \chi \equiv 1$, we have
\begin{equation}
\label{eq:delta2}
\begin{aligned}
&\tfrac12 \| u \|^2_{H^{-N}_{\Lambda_\e}}  + C_1^2+
2  C_0^2\\
& \geq \langle \chi^2  \langle
\xi \rangle^{m-1}  \langle \xi\rangle^{-N}  T_{\Lambda_\e} u,  \langle \xi\rangle^{-N}T_{\Lambda_\e} u \rangle_{ L^2_{\Lambda_\e}} \\
&\qquad\qquad\qquad-\langle M(1-\chi^2)\langle \xi\rangle^{m-1}\langle \xi\rangle^{-N}T_{\Lambda_\e} u,\langle \xi\rangle^{-N}T_{\Lambda_\e} u\rangle_{L^2_{\Lambda_\e}}\\& 
\geq \langle \langle
\xi \rangle^{m-1}  \langle \xi\rangle^{-N}  T_{\Lambda_\e} u,  \langle \xi\rangle^{-N}T_{\Lambda_\e} u \rangle_{ L^2_{\Lambda_\e}} -(M+1)\|u\|_{H^{-N + \frac{m-1}{2}}},
 \end{aligned} 
\end{equation}
where in the last line we use that $\chi\equiv 1$ on $\supp G_\e$.

Using $m\geq 1$ and rearranging, this yields
\[  \| u \|_{H^{-N}_{\Lambda_\e} }^2 \leq 2C_1^2+4C_0^2+2(M+1)\|u\|_{H^{-N + \frac{m-1}{2}}}. \]
where $ C_1, C_0$ and $ M $ are constants {\em independent of $\e$}. 

Since $ {\Lambda_\e} \cap \{ | \xi| < 1/\e \} = \Lambda_0 \cap 
\{ |\xi | < 1/\e \} $ where $G_0:=\Phi|\xi|$, we have that $ H_\epsilon |_{ | \xi| < 1/\e } 
= H_0 |_{ | \xi| < 1/\e } $,  where $ H_\epsilon
= \theta \xi \partial_\xi G_\epsilon + \theta G $ is the corresponding weight. Therefore, the
monotone convergence
theorem implies that
$ u \in H_{\Lambda_0}  $.  Since
 $ \Phi ( x_0, t\xi_0 ) = 1 $, $ t \gg 1 $, Proposition \ref{p:awf}
shows that $ ( x_0 , \xi_0 ) \notin \WFa ( u ) $.
\end{proof}

\medskip\noindent\textbf{Acknowledgements.}
Partial support for M.Z. by the National Science Foundation grant DMS-1500852 and for J.G. by the National Science Foundation grant DMS-1900434 is also  
gratefully acknowledged.

\end{document}